\documentclass[final, 12pt]{amsart}

\usepackage[dvipsnames]{xcolor}
\usepackage{etoolbox}
\patchcmd{\section}{\normalfont}{\normalfont\color{darkblue}}{}{}
\patchcmd{\subsection}{\normalfont}{\normalfont\color{darkblue}}{}{}
\patchcmd{\subsubsection}{\normalfont}{\normalfont\color{darkblue}}{}{}
\usepackage{lipsum}
\usepackage{amssymb, amsmath, amsthm,geometry}
\usepackage{mathrsfs} 
\usepackage[T1]{fontenc}
\usepackage[english]{babel}
\usepackage{lmodern}

\usepackage{geometry}

\usepackage{bm}
\usepackage{subfig,tikz}
\usepackage{wrapfig}
\usepackage{color}
	\definecolor{darkblue}{rgb}{0.0, 0.0, 0.55}
\usepackage{xspace}
\usetikzlibrary{shapes.geometric}

\usepackage{graphicx}
\usepackage[colorlinks=true, citecolor=green, linkcolor=red, urlcolor=red]{hyperref}
 \usepackage{xcolor}
\usepackage{hyperref}

\usepackage{lineno}

\newgeometry{asymmetric, centering}
\usepackage{ifdraft}
\ifoptionfinal{
\usepackage[disable]{todonotes}
}{
\usepackage[norefs, nocites]{refcheck}

\usepackage[bordercolor=white, color=white]{todonotes}
}
\usepackage[bottom]{footmisc}

\makeatletter\providecommand\@dotsep{5}\def\listtodoname{List of Todos}\def\listoftodos{\hypersetup{linkcolor=black}\@starttoc{tdo}\listtodoname\hypersetup{linkcolor=blue}}\makeatother

\newtheorem{lemma}{Lemma}
\newtheorem{proposition}{Proposition}
\newtheorem{theorem}{Theorem}

\newtheorem{definition}{Definition}
\theoremstyle{remark} 

\newtheorem{remark}{Remark}
\theoremstyle{question}
\newtheorem{question}{Question}

\newcommand{\bel}{\begin{equation} \label}
\newcommand{\ee}{\end{equation}}
\def\beq{\begin{equation}}
\def\eeq{\end{equation}}
\newcommand{\bea}{\begin{eqnarray}}
\newcommand{\eea}{\end{eqnarray}}
\newcommand{\beas}{\begin{eqnarray*}}
\newcommand{\eeas}{\end{eqnarray*}}
\newcommand{\pd}{\partial}

\def\C{\mathbb C}

\def\R{\mathbb R}

\def\Z{\mathbb Z}
\def\N{\mathbb N}
\def\M{\mathcal M}
\def\g{\textsl{g}}

\def\CI{\mathcal C}
\renewcommand{\leq}{\leqslant}
\renewcommand{\geq}{\geqslant}

\def\p{\partial}

\DeclareMathOperator{\Tr}{\textit{Tr}}

\title[A Semi-linear equation in Riemannian Geometry]{An inverse problem for a semi-linear Elliptic Equation in Riemannian Geometries}

\author[]{Ali Feizmohammadi}
\address[]{Department of Mathematics, University College London, London, UK-WC1E  6BT, United Kingdom}
\email{a.feizmohammadi@ucl.ac.uk}
\author[]{Lauri Oksanen}
\address[]{Department of Mathematics, University College London, London, UK-WC1E  6BT, United Kingdom}
\email{l.oksanen@ucl.ac.uk}

\begin{document}
\maketitle
\begin{abstract}
We study the inverse problem of unique recovery of a complex-valued scalar function $V:\M \times \C\to \C$, defined over a smooth compact Riemannian manifold $(\M,\textsl{g})$ with smooth boundary, given the Dirichlet-to-Neumann map, in a suitable sense, for the elliptic semi-linear equation $-\Delta_{\textsl{g}}u+V(x,u)=0$. We show that uniqueness holds for a large class of non-linearities when the manifold is conformally transversally anisotropic. The proof is constructive and is based on a multiple-fold linearization of the semi-linear equation near complex geometric optic solutions for the linearized operator and the resulting non-linear interactions. These interactions result in the study of a weighted integral transform along geodesics, that we call the Jacobi weighted ray transform. 
\end{abstract}
{
  \hypersetup{linkcolor=black}
  \tableofcontents
}

\section{Introduction} 
\label{sec1}
Let $(\M,\textsl{g})$ be a smooth compact Riemannian manifold with a smooth boundary $\pd \M$ and $\dim \M:=n\geq 3$. Let $\alpha \in (0,1)$ and consider an a priori unknown function $V:\M \times \C \to \C$. We make the following standing assumptions.
\begin{itemize}
\label{nonlinear_ass}
\item[(i)]{$V(\cdot,z) \in \CI^{\alpha}(\M),\quad \forall z \in \C$,} 
\item[(ii)]{$V(x,0)=0,\quad \forall x \in \M$,}
\item[(iii)]{$V$ is analytic with respect to $z$ in the $\CI^{\alpha}(\M)$ topology,}
\end{itemize}
where $\CI^{\alpha}(\M)$ is the space of H\"{o}lder continuous complex-valued functions with exponent $\alpha$. By analyticity with respect to $z \in \C$ we mean that the following limit exists in the $\CI^{\alpha}(\M)$ topology,
$$\pd_z V(x,z):=\lim_{h \to 0} \frac{V(x,z+h)-V(x,z)}{h}.$$
As a result of analyticity, the function $V$ admits a power series representation in the $C^{\alpha}(\M)$ topology given by the expression
\bel{power_series} V(x,z)=\sum_{k=1}^\infty V_k(x)\, \frac{z^k}{k!},\ee
where $V_k(x):=\pd^k_z V(x,0)\in \CI^{\alpha}(\M)$. We additionally impose the following conditions on the set of admissible functions $V(x,z)$:
\begin{itemize}
\item[(iv)]{$0$ is not a Dirichlet eigenvalue for the operator $-\Delta_{\g}+V_1(x)$ on $(\M,\g)$.} 
\end{itemize}
Here, $\Delta_{\textsl{g}}$ denotes the Laplace-Beltrami operator on $(\M,\textsl{g})$ given in local coordinates by the expression $\Delta_{\textsl{g}}=\sum_{j,k=1}^{n} \frac{1}{\sqrt{\textsl{g}}}\frac{\pd}{\pd x^j}\left(\sqrt{\textsl{g}}\, \textsl{g}^{jk}\frac{\pd }{\pd x^k} \right).$ 
\\

In this paper, we consider the semi-linear elliptic equation 
\bel{pf1}
\begin{aligned}
\begin{cases}
-\Delta_{\textsl{g}} u+V(x,u)=0, 
&\forall x \in \M
\\
u= f \in B^{\alpha}_{r_0}(\pd \M),
&\forall x \in \pd \M
\end{cases}
    \end{aligned}
\ee
where $B^{\alpha}_{r_0}(\pd\M):=\{h\in \CI^{2,\alpha}(\pd \M)\,|\,\|h\|_{\CI^{2,\alpha}(\pd \M)}\leq r_0\}$. In Section~\ref{direct}, we show that, given fixed $r_0,r_1>0$ sufficiently small, equation \eqref{pf1} admits a unique solution $u \in B^{\alpha}_{r_1}(\M)$. Moreover, there exists a constant $C>0$ depending only on $r_0,r_1$ such that  
\bel{dirichletcont}\|u\|_{\CI^{2,\alpha}{(\M)}}\leq C \|f\|_{\CI^{2,\alpha}(\pd \M)} \quad \forall f \in B^{\alpha}_{r_0}(\pd \M).\ee
We subsequently define the Dirichlet-to-Neumann (DN) map, $\Lambda_{V}$, for equation \eqref{pf1} through the expression
\bel{DNsemi}
\CI^{1,\alpha}(\pd\M)\ni \Lambda_{V} f:= \pd_\nu u |_{\pd \M}, \quad \forall f \in B^{\alpha}_{r_0}(\pd\M), 
\ee
where $\nu$ denotes the unit outward normal vector field on $\pd \M$. This paper is concerned with the following question.
\begin{question}{}{}
\label{ques}
Given the map $\Lambda_{V}$, can one uniquely determine the function $V$?
\end{question}
We will briefly review the history related to inverse problems for non-linear elliptic equations in Section~\ref{literature}. For now, let us recall some facts about the case where $V(x,z)\equiv V_1(x)z$. In this case, the problem reduces to a version of the Calder\'{o}n conjecture \cite{MR590275}.  This formulation of the conjecture has been extensively studied but remains open in general geometries $(\M,\g)$ with dimension $n \geq 3$. Uniqueness of the coefficient $V_1$ has been proved for analytic metrics with an analytic function $V_1$ \cite{MR1029119}, the Euclidean metric \cite{MR1370758,MR873380} and the hyperbolic metric \cite{10.2307/40067898}. Beyond these cases, the most general uniqueness result is obtained in the so-called {\em conformally transversally anisotropic} (CTA) geometries defined as follows.
\bigskip
\begin{definition}
\label{CTA}
Let $(\M,\g)$ be a compact oriented smooth Riemannian manifold with smooth boundary and dimension $n$. We say that $(\M,\g)$ is conformally transversally anisotropic, if $n\geq3$ and the following embedding holds:
$$\M\subset I^{\text{int}} \times M^{\text{int}}\quad\text{and}\quad \g(x^0,x')=c(x^0,x')((dx^0)^2\oplus g(x')),$$ 
where $I$ is a finite interval, $c(x^0,x')>0$ is a smooth function and $(M,g)$ is a smooth compact orientable manifold of dimension $n-1$ with a smooth boundary $\pd M$. 
\end{definition}
\bigskip
In \cite{DosSantosFerreira2009} it was proved that in the linear case $V(x,z)=V_1(x)z$, the Dirichlet-to-Neumann map $\Lambda_{V}$ uniquely determines a bounded function $V_1$, under the strong assumption that the transversal manifold is {\em simple}, that is to say $(M,g)$ has a strictly convex boundary and given any two points in $M$ there exists a unique geodesic connecting them. This result was subsequently strengthened in \cite{MR3562352} where the authors showed that $\Lambda_{V_1}$ uniquely determines $V_1$, if the geodesic ray transform is injective on the transversal manifold. The inversion of the geodesic ray transform is open in general, and has only been proved under certain geometrical assumptions, see for example the discussion in \cite[Section 1]{MR3562352}. For a broad review of the Calder\'{o}n conjecture, and alternative formulations with the presence of non-linear coefficients, we refer the reader to survey articles \cite{MR3460047,SLSEDP_2012-2013____A13_0}.
\subsection{Main results}
Let us return to Question~\ref{ques}. We will consider only the case where $(\M,\g)$ is a CTA manifold. Before stating our results let us briefly review some notations for geodesic dynamics on $(M,g)$. Let $SM\subset TM$ denote the unit sphere bundle on $(M,g)$ and $\gamma(\cdot,x,\theta)$ be the unit speed geodesic with initial data $(x,\theta)$. For all $(x,\theta) \in SM^{\text{int}}$, we define the exit times
\bel{tau}\tau_{\pm} =\sup \,\{r>0\,|\, \gamma(\pm r;x,\theta) \in \pd M,\quad \dot{\gamma}(\pm r;x,\theta) \notin T\pd M\},\ee
and subsequently call a geodesic $\gamma$ to be maximal, if and only if $\tau_{\pm}<\infty$. Next, we define an admissibility condition on the transversal manifold $(M,g)$ as follows.
\begin{definition}
\label{adm}
Let $(M,g)$ be a smooth compact Riemannian manifold with boundary. We say that $(M,g)$ is admissible if there exists a dense set of points $\mathscr T \subset M$ such that given any point $p \in \mathscr T$ there exists a non-self-intersecting maximal geodesic $\gamma$ through $p$ that contains no conjugate points to $p$.
\end{definition}
The first result in this paper can now be stated as follows.
\begin{theorem}
\label{t1}
Let $(\M,\g)$ be a CTA manifold such that the transversal manifold $M$ is admissible. Suppose that $V(x,z)$ satisfies conditions (i)--(iv), that $V_1$ is smooth and that $V_1, V_2$ are a priori known. Then, the Dirchlet-to-Neumann map $\Lambda_{V}$ uniquely determines the function $V$.
\end{theorem}

The proof of this theorem relies on a multiple-fold linearization of \eqref{pf1} that results in the interaction of the so called complex geometric optic solutions for the corresponding linearized equation. Since $V_1$ is assumed to be known, the complex geometric optic solutions will be known as well. The smoothness assumption on $V_1$ is imposed in order to make these solutions smooth and also to simplify the task of proving suitable decay rates (see Proposition~\ref{cgo_sol}). Under the assumption that $V_2$ is assumed to be known, the non-linear interaction of the complex geometric optic solutions will result in a weighted ray transform along geodesics on the transversal manifold $M$. This weighted transform will be shown to be invertible along a single geodesic (see Proposition~\ref{jacobi ray inversion}).  

Our second main result is concerned with the recovery of the function $V$ without imposing the assumption that the coefficient $V_2$ is known, in the cases where the manifold is three or four dimensional. 
\begin{theorem}
\label{t2}
Let $(\M,\g)$ be a three or four dimensional CTA manifold such that given any point on the transversal manifold $M$ there exists a maximal non-self-intersecting geodesic without conjugate points through that point. Suppose that $V(x,z)$ satisfies conditions (i)-(iv) and that $V_1$ is a priori known and smooth. Then the Dirichlet-to-Neumann map $\Lambda_V$ uniquely determines the function $V$.
\end{theorem}
The proof of this theorem mostly follows the same technique as the previous theorem. However, due to the weaker assumption on the coefficient $V$, namely that $V_2$ is unknown, the non-linear interaction of the complex geometric optic solutions results in a different ray transform along geodesics on the transversal manifold $M$. The inversion of this transform along a single geodesic is proved when the transversal manifold is two or three dimensional and left open in higher dimensions (see Proposition~\ref{jacobi ray inversion 1}). We also refer the reader to Remark~\ref{rmk_jacobi} in Section~\ref{section_jacobi} where the restriction to three and four dimensions is discussed further.

\subsection{Previous literature}
\label{literature}
The study of non-linear partial differential equations is an interesting topic in its own right, due to the complexity of the subject matter and as such, the corresponding inverse problems also carry significant mathematical interest. However, let us point out that there are applications for these inverse problems outside the realm of mathematics as well. Indeed, a large class of inverse problems for elliptic nonlinear equations can be seen as the study of stationary solutions to nonlinear equations describing physical phenomena. For example, we mention the nonlinear Schr\"{o}dinger equation that arises as nonlinear variations of the classical field equations and has applications in the study of nonlinear optical fibers, planar wave guides and Bose Einstein condensates \cite{Boris_2005}. Other examples include nonlinear Klein-Gordon or Sine-Gordon equations with applications to the study of general relativity \cite{NAKAMURA2014445} and relativistic super-fluidity \cite{Waldron_2017} respectively.

The majority of the literature dealing with inverse problems for non-linear elliptic equations is in the Euclidean geometry. The first uniqueness result was obtained by Isakov and Sylvester in \cite{MR1295934} where the authors considered a Euclidean domain of dimension greater than or equal to three with non linear functions $V(x,u)$ that satisfy the homogeneity property (ii), and showed that under a monotonicity condition for $V$ and suitable bounds on $V$, $\p_u V$ and $\p^2_{u}V$, the non-linearity can be uniquely recovered on a specific subset of $\M \times \R$. There, it was also proved that under a stronger bound on $V$, it could be recovered everywhere. Removing the homogeneity property (ii) introduces a natural gauge for the uniqueness of the non-linearity. This was studied by Sun in \cite{MR2602870} under similar smoothness and monotonicity assumptions. There, a similar uniqueness result as in \cite{MR1295934} was proved (up to the natural gauge), under the additional assumption that a common solution exists. 

In dimension two, the problem was first solved by Sylvester and Nachman in \cite{MR1311909}, where the authors considered a domain in two-dimensional Euclidean space with a Carath\'{e}odory type non-linearity that has a continuous bounded $L^p$-valued derivative in the $u$ variable and proved unique recovery of the non-linearity. In \cite{MR2506853} uniqueness is proved for yet another family of admissible non-linearities in two dimensional Euclidean domains. There, a connection is also made between the theoretical study of these types of semi-linear inverse problems and the physical study of semi-conductor devices and ion channels. We also mention the work of Imanuvilov and Yamamoto in \cite{IY13} where the authors considered the partial data problem for the operator $\Delta u + q(x)u+ V(x,u)$ on arbitrary open subsets of the boundary in two dimensions. There it was shown that if $V(x,0)=\p_u V(x,0)=0$, it is possible to uniquely recover $q$ everywhere and also that it is possible to recover $V$ in certain subsets of the domain, under suitable bounds on the non-linear function $V$.

Aside from the study of inverse problems for semi-linear equations in Euclidean geometries, let us also mention that there are several works related to inverse problems for quasi-linear elliptic equations (see for example \cite{2019arXiv190307034C,EPS14,Isakov01,2018arXiv180609586M,MR1376299,MR2053845,ZU}). It should be emphasized that the key idea in all of these results has been a linearization technique introduced by Isakov in \cite{MR1233645} in the context of semi-linear parabolic equations and developed further in \cite{MR1861472,MR1311909,MR1295934,MR1376299,MR2053845}. This linearization technique together with the uniqueness results for the Calder\'{o}n conjecture in Euclidean domains leads to the unique recovery of the non-linear terms. 

The main novelty of this paper is to extend uniqueness results for non-linear elliptic equations to a wider class of Riemannian manifolds, known as conformally transversally anisotropic manifolds (see Definition \ref{CTA}). We consider local solutions about the trivial solution, but our proof is based on a multiple-fold linearization technique that differs from most of the previously mentioned works. As already discussed, the results in the Euclidean setting rely on the fact that uniqueness holds for the linearized inverse problem. This is no longer the case when $\M$ is assumed to be conformally transversally anisotropic. Indeed, uniqueness results for the linearized problem rely on injectivity of the geodesic ray transform on $(M,g)$ that is known to be true under strong geometric assumptions such as simplicity of the transversal manifold $(M,g)$ or existence of a strictly convex foliation \cite{UV17}. The strength of our results lies in removing such strong geometric assumptions. On the other hand, contrary to the Euclidean cases, the results here assume analyticity of $V(x,u)$ with respect to $u$. 

The multiple-fold linearization technique in this paper is inspired by the study of similar types of non-linear problems for hyperbolic equations that was developed by Kurylev, Lassas and Uhlmann in \cite{2014arXiv1406.4776K, KLU18} in the context of Einstein scalar field equations and used in subsequent works in the context of semi-linear wave equations (see for example \cite{CLOP,FO19,LUW17,Lassas2018,WZ}). However, these works are based on the study of propagation of singularities for linear wave equations and the non-linear interactions of these singularities, making it difficult to apply them to an elliptic problem. Another key difference with all previous works in the hyperbolic setting is that we study non-linear interaction of localized solutions that correspond to a single geodesic. This will lead us to the study of a weighted transform along geodesics that we call the Jacobi ray transforms of the first and second kind. We show that it is possible to invert these transforms along a single geodesic (see Propositions~\ref{jacobi ray inversion 1}--\ref{jacobi ray inversion}).

We conclude this introductory section by remarking that while writing this
paper we became aware of an upcoming preprint by Matti Lassas, Tony Liimatainen,
Yi-Hsuan Lin and Mikko Salo, which simultaneously and independently proves
a similar result. We agreed to post our respective preprints to arXiv
at the same time. See \cite{LLLS} for their preprint.
\subsection{Outline}
This paper is organized as follows. Section~\ref{prelim} is concerned with some preliminary discussions. We show that the Dirichlet-to-Neumann map $\Lambda_{V}$ (see \eqref{DNsemi}) is well-defined. We also discuss the linearization method for solutions to equation \eqref{pf1} near the trivial solution, in particular showing the appearance of what we call a nonlinear interaction. The rest of Section~\ref{prelim} is concerned with some lemmas and notations that will be needed throughout the paper. In Section~\ref{section_jacobi} we define the Jacobi weighted transform of the first and second kind along a fixed geodesic, and subsequently prove injectivity results for these two transforms, see Propositions~\ref{jacobi ray inversion 1}--\ref{jacobi ray inversion}. Section~\ref{cgo_section} starts with a review of the well-known Gaussian quasi modes for the linearized operator following \cite{MR3562352}. In the remainder of this section we use this construction, together with a Carleman estimate to produce a family of complex geometric optic solutions for the linearized operator. In Section~\ref{section_uniqueness}, we use an induction argument, based on the application of our linearization technique near the complex geometric optic solutions, to complete the proof of Theorem~\ref{t1}. Section~\ref{section_t2} is concerned with the proof of Theorem~\ref{t2}.

\section{Preliminaries}
\label{prelim}
\subsection{Direct problem}
\label{direct}
In this section we prove the following proposition for the direct problem \eqref{pf1}.
\begin{proposition}
\label{direct_problem}
$\exists r_0,r_1>0$ depending on $(\M,\g)$, such that equation \eqref{pf1} admits a unique solution $u \in B^{\alpha}_{r_1}(\M)$. Moreover, there holds
$$ \|u\|_{\CI^{2,\alpha}(\M)}\leq C \|f\|_{\CI^{2,\alpha}(\pd \M)},\quad \forall f \in B^{\alpha}_{r_0}(\pd \M),$$
for some constant $C$ that depends on $(\M,\g)$, $r_0$ and $r_1$.
\end{proposition}

Let us define the Schr\"{o}dinger operator $\mathcal P_{V_1}=-\Delta_{\g}+V_1(x)$, and consider the linear equation
\bel{pf0}
\begin{aligned}
\begin{cases}
\mathcal P_{V_1}\, u=F,  &\forall x \in \M \\
u= f  &\forall x \in \pd \M
\end{cases}
    \end{aligned}
\ee
where $(f,F) \in \CI^{2,\alpha}(\pd \M)\times \CI^{\alpha}(\M)$. We introduce the solution operators $\mathcal G_{V_1}^D,\mathcal G_{V_1}^S$ so that the function
$\mathcal G_{V_1}^D f$ is the unique solution to \eqref{pf0} subject to $F \equiv 0$ and $\mathcal G_{V_1}^S F$ is the unique solution to \eqref{pf0} subject to $f \equiv 0$. There is a constant $\kappa>0$ (see for example \cite[Chapter 4.4]{MR1814364}) such that
\bel{greenfcn}
\|\mathcal G^D_{V_1}\|_{\CI^{2,\alpha}(\pd\M)\to\CI^{2,\alpha}(\M)} + \|\mathcal G^S_{V_1}\|_{\CI^{\alpha}(\M)\to\CI^{2,\alpha}(\M)} \leq \kappa.
\ee
Let us now define the function $\tilde{V}(x,z):= V(x,z)-V_1(x)z$. We have the following lemma.
\begin{lemma}
\label{direct_1}
Let $r \in (0,1)$. Given any $u_0,u_1 \in B^{\alpha}_r(\M)$, the following estimates hold.
\begin{itemize}
\item[(i)]{$\|\tilde{V}(x,u_0(x))\|_{\CI^{\alpha}(\M)}\leq \tilde{\kappa}\|u_0\|^2_{\CI^{\alpha}(\M)}$,}
\item[(ii)]{$\|\tilde{V}(x,u_1(x))-\tilde{V}(x,u_0(x))\|_{\CI^{\alpha}(\M)}\leq \tilde{\kappa}r\|u_1-u_0\|_{\CI^{\alpha}(\M)}$,}
\end{itemize}
where $\tilde{\kappa}>0$ is independent of $r$.
\end{lemma}
\begin{proof}
First, observe that $\CI^{\alpha}(\M)$ is closed under multiplication and that there exists a constant $C>0$, depending on $(\M,\g)$, such that for any $v,w \in \CI^{\alpha}(\M)$ there holds
$$ \|vw\|_{\CI^{\alpha}(\M)}\leq C \|v\|_{\CI^{\alpha}(\M)}\|w\|_{\CI^{\alpha}(\M)}.$$
Now, using the fact that $\tilde{V}(x,0)=\pd_z \tilde{V}(x,0)=0$ we write
$$ \tilde{V}(x,u_k(x))= \frac{1}{2}\int_{C} \pd^2_{z}\tilde{V}(x,z) (u_k(x)-z)\,dz,\quad k=0,1,$$
where $C$ is a path connecting $0$ to $u(x)$ and the integral is in the sense of the $\CI^\alpha(\M)$ norm limits of the Riemann partial sums. Applying the $C^\alpha$ norm we deduce that
\[
\|\tilde{V}(x,u_k(x))\|_{\CI^{\alpha}(\M)}\leq \left(C_1\sup_{|z|\leq r} \|\pd^2_z \tilde{V}(x,z)\|_{\CI^\alpha(\M)}\right)\|u_k\|^2_{\CI^\alpha(\M)}.
\]
Similarly we have,
\[
\|\tilde{V}(x,u_1(x))-\tilde{V}(x,u_0(x))\|_{\CI^{\alpha}(\M)}\leq \left(C_2\sup_{|z|\leq r}\|\pd^2_z \tilde{V}(x,z)\|_{\CI^{\alpha}(\M)}\right)r\|u_1-u_0\|_{\CI^{\alpha}(\M)}.
\]
for some $C_1,C_2>0$. Finally, using smoothness of $V(x,z)$ with respect to $z$, we deduce that 
$$ \sup_{|z|\leq r}\|\pd^2_z \tilde{V}(x,z)\|_{\CI^{\alpha}(\M)}\leq C_3,$$
for some constant $C_3>0$ independent of $r$, since $0<r<1$. The claim follows immediately by combining the preceding three bounds.
\end{proof}
\begin{proof}[Proof of Proposition~\ref{direct_problem}]
We start by fixing 
$$r_0<\min\{\frac{1}{1+3\kappa},\frac{1}{4\tilde{\kappa}\kappa(1+\kappa^2)},\frac{1}{\tilde{\kappa}(2\kappa+1)},\frac{1}{\tilde{\kappa}(\kappa+1)^2}\},\quad r_1=(\kappa+1) r_0.$$ 
First we show existence of a solution $u \in B^{\alpha}_{r_1}(\M)$. Write $u=\mathcal G_{V_1}^D\,f+\tilde{u}$ and observe that there exists a one to one correspondence between $\CI^{2,\alpha}(\M)$ solutions to equation \eqref{pf1} and solutions to the integral equation
\bel{integral_pf}
\tilde{u}= -\mathcal G^S_{V_1} \left( \tilde{V}(x,\mathcal G_{V_1}^D\,f+\tilde{u})\right)=:T_f(\tilde{u}).
\ee
Next noting that $r_0<1$, we may apply Lemma~\ref{direct_1}, and this together with the bound \eqref{greenfcn} yields
\bel{T_f}\|T_f v\|_{\CI^{2,\alpha}(\M)} \leq 2 \tilde{\kappa} \kappa\left(\|v\|_{\CI^{\alpha}(\M)}^2+\|\mathcal G_{V_1}^D\,f\|^2_{\CI^{\alpha}(\M)} \right)\quad \forall v \in B_{r_0}^{\alpha}(\M).\ee
Applying \eqref{greenfcn} again and noting that $r_0<\frac{1}{4\kappa\tilde{\kappa}(\kappa^2+1)}$, we deduce that $T_f$ maps the closed set $\mathcal B^{\alpha}_{r_0}(\M)$ to itself. Additionally, one can verify in the same way that $T_f$ is a contraction mapping on $B^{\alpha}_{r_0}(\M)$. The Banach fixed point theorem applies and we conclude that there exists a solution $\tilde{u}\in B^{\alpha}_{r_0}(\M)$ to equation \eqref{integral_pf}. Observe subsequently that $u \in \CI^{2,\alpha}(\M)$ defined above solves \eqref{pf1}. Applying \eqref{T_f} we have
\[
\|\tilde{u}\|_{\CI^{2,\alpha}(\M)}\leq 2 \tilde{\kappa} \kappa\left(\|\tilde{u}\|_{\CI^{\alpha}(\M)}^2+\|\mathcal G_{V_1}^D\,f\|^2_{\CI^{\alpha}(\M)} \right)\leq \frac{1}{2} \|\tilde{u}\|_{\CI^{2,\alpha}(\M)}+2\tilde{\kappa}\kappa^3\|f\|_{\CI^{2,\alpha}(\pd \M)}^2,
\]
Thus yielding the continuity estimate $\|u\|_{\CI^{2,\alpha}(\M)} \leq (\kappa+1)\|f\|_{\CI^{2,\alpha}(\pd\M)}.$ This latter estimate also shows that $u \in B^{\alpha}_{r_1}(\M)$. 

Next we show uniqueness. Suppose for contrary that $u_1,u_2 \in B_{r_1}^{\alpha}(\M)$ with $u_1\neq u_2$ solving equation \eqref{pf1}. Define $\tilde{u}_k$ for $k=1,2$ as above and note that $\tilde{u}_k=T_f \tilde{u}_k$.  Since $r_1<1$, Lemma~\ref{direct_1} applies to obtain
$$\|\tilde{u}_k\|_{\CI^{2,\alpha}(\M)}\leq \|u_k\|_{\CI^{2,\alpha}(\M)}+\|\mathcal G_{V_1}^Df\|_{\CI^{2,\alpha}(\M)}\leq r_2,$$
where $r_2=(1+2\kappa)r_0$. Finally, since $r_0<\frac{1}{1+3\kappa}$, we can apply Lemma~\ref{direct_1} again to deduce that $T_f$ is a contraction mapping on the set $B_{r_2}^{\alpha}(\M)$. Therefore 
$$\|\tilde{u}_1-\tilde{u}_2\|_{\CI^{2,\alpha}(\M)}=\|T_f\tilde{u}_1-T_f\tilde{u}_2\|_{\CI^{2,\alpha}(\M)}\leq \tilde{\kappa}(2\kappa+1) r_0\|\tilde{u}_1-\tilde{u}_0\|_{\CI^{2,\alpha}(\M)}<\|\tilde{u}_1-\tilde{u}_0\|_{\CI^{2,\alpha}(\M)},$$
which is a contradiction.
\end{proof}
\subsection{Multiple-fold linearization method}
\label{section_lin}
We have established that the forward problem \eqref{pf1} is well-posed, and therefore also the DN map \eqref{DNsemi} is well-defined. In order to prove Theorems~\ref{t1}--\ref{t2}, we will work with families of Dirichlet datum $f$ that will be arbitrarily small with respect to the $\CI^{2,\alpha}(\pd\M)$ norm, and are therefore only interested in the behavior of $\Lambda_V$ near $f\equiv 0$. To set this idea in motion, let $m \in \N$ and consider a parameter $\varepsilon=(\varepsilon_1,\ldots,\varepsilon_m) \in \C^m$ and a family of Dirichlet datum \bel{lin_data}f_\varepsilon=\sum_{k=1}^m \varepsilon_k f_k,\ee
where $f_k \in \CI^{2,\alpha}(\pd\M)$ are fixed. Clearly, for all $|\varepsilon|$ sufficiently small, there exists a unique solution $u_\varepsilon$ to equation \eqref{pf1} subject to Dirichlet data $f_\varepsilon$.

\subsubsection{Analytic dependence on $\varepsilon$}
Next, we prove that $u_\varepsilon$ is analytic in a neighborhood of $\varepsilon=0$, that is to say $u_\varepsilon$ admits a power series representation with respect to the parameter $\varepsilon$ in the $\mathcal C^{2,\alpha}(\M)$ topology. It suffices to show that $u_\varepsilon$ is analytic with respect to each $\varepsilon_k$ for $k=1,\ldots,m$ (see for example \cite[Theorem 1.2.25]{VS}). To this end we prove that given a fixed family $\{f_k\}_{k=1}^m$ and any fixed $\varepsilon$ in a sufficiently small neighborhood of the origin in $\C^m$, the limit 
\bel{lim_comp}\lim_{h\to0}\frac{u_{\varepsilon+he_l}-u_{\varepsilon}}{h}\ee 
exists in the $\CI^{2,\alpha}(\M)$ sense, where $h\in \C$ and $e_l$ denotes the $l^{\text{th}}$ unit vector in $\C^m$ with $l=1,\ldots,m$. 

As a first step, we proceed to prove that for all $|\varepsilon|$ small enough and all $|h|<|\varepsilon|$ there holds
\bel{c1bound}
\|u_{\varepsilon+he_l}-u_\varepsilon\|_{\mathcal C^{2,\alpha}(\M)}\leq C |h|,
\ee
where $C>0$ is independent of $\varepsilon$ and $h$. We begin by observing that for $|\varepsilon|$ small and all $|h|<|\varepsilon|$, we can apply Proposition~\ref{direct_problem} to obtain the estimate 
\bel{epsilon_bound_lem}   
\|u_\varepsilon\|_{\CI^{2,\alpha}(\M)} + \|u_{\varepsilon+he_l}\|_{\CI^{2,\alpha}(\M)} \leq C |\varepsilon|,
\ee
for some constant $C>0$ independent of $\varepsilon$ and $h$. Next, we use equation \eqref{integral_pf} to write
\begin{align*}
u_{\varepsilon}&= \mathcal G^D_{V_1} f_\varepsilon -\mathcal G^S_{V_1}(\tilde{V}(x,u_\varepsilon)),\\
u_{\varepsilon+he_l}&=\mathcal G^D_{V_1} f_{\varepsilon+he_l} -\mathcal G^S_{V_1}(\tilde{V}(x,u_{\varepsilon+he_l})).
\end{align*}
Subtracting these two equations and applying (ii) in Lemma~\ref{direct_1} together with \eqref{epsilon_bound_lem}, it follows that 
$$\|u_{\varepsilon+he_l}-u_\varepsilon\|_{\CI^{2,\alpha}(\M)} \leq C \left( |h| \|f_l\|_{\CI^{2,\alpha}(\p\M)} + |\varepsilon| \|u_{\varepsilon+he_l}-u_\varepsilon\|_{\CI^{2,\alpha}(\M)}\right),$$
where $C>0$ is independent of $h$ and $\varepsilon$. Finally, the bound \eqref{c1bound} follows from this estimate for $\varepsilon$ sufficiently small.

Next, we proceed to show the main claim that the limit in \eqref{lim_comp} exists. Since $0$ is not a Dirichlet eigenvalue for $\mathcal P_{V_1}$, it follows that the same is true for the operator $\mathcal P_{\pd_z V(\cdot,u_\varepsilon)}$, given that $|\varepsilon|$ is sufficiently small. We subsequently define $\mathcal G^D_\varepsilon$ and $\mathcal G^S_\varepsilon$ analogously to $\mathcal G^D_{V_1}$ and $\mathcal G^S_{V_1}$, corresponding to equation \eqref{pf0} with potential $\pd_z V(\cdot,u_\varepsilon)$ in place of $V_1$. We write
$$\tilde{u}_{h}:=u_{\varepsilon+he_l}-u_{\varepsilon}-h\,\mathcal G_\varepsilon^D f_{l}.$$
The function $\tilde{u}_{h}$ satisfies the equation
$$ \tilde{u}_h=-\mathcal G_{\varepsilon}^S\left(\tilde{V}_{\varepsilon,h}(x)\right)$$
where 
$$\tilde{V}_{\varepsilon,h}(x)=V(x,u_{\varepsilon+he_l}(x))-V(x,u_\varepsilon(x))-\pd_z V(x,u_\varepsilon(x))(u_{\varepsilon+he_l}-u_\varepsilon).$$
Using the smoothness of $V(x,z)$ with respect to $z$ and analogously to Lemma~\ref{direct_1} we deduce that there exists $C_\varepsilon>0$ such that
\bel{est_V} \|\tilde{V}_{\varepsilon,h}(x)\|_{\mathcal C^{\alpha}(\M)}\leq C_\varepsilon \|u_{\varepsilon+he_l}-u_\varepsilon\|^2_{\mathcal C^{\alpha}(\M)}.\ee
Thus, for $h$ sufficiently small, by using the bounds \eqref{greenfcn}, \eqref{c1bound} and \eqref{est_V}, we obtain 
$$ \|\tilde{u}_h\|_{\CI^{2,\alpha}(\M)} \leq C_\varepsilon |h|^2$$
for some $C>0$ independent of $h$. Hence,
\bel{anal_epsilon}\lim_{h\to0}\frac{u_{\varepsilon+he_l}-u_{\varepsilon}}{h}=\mathcal G_\varepsilon^D f_l\ee
holds in the $\CI^{2,\alpha}$ sense, proving that $u_\varepsilon$ depends analytically in each of its parameters $\varepsilon_l$ in a neighborhood of $\varepsilon=0$ in $\C^m$. 
\\
\subsubsection{Non-linear interaction of linearized solutions}
Let us now use this linearization technique to first show that $\Lambda_{V}$ determines the Dirichlet to Neumann map $\Lambda^{\text{lin}}_{V_1}$ associated to the linear operator $\mathcal P_{V_1}$. Of course this is a somewhat redundant argument as $V_1$ will be assumed to be known for us, but nevertheless this simple case will shed some light on the higher order linearization arguments. Let $m=1$ so that $\varepsilon \in \C$ and write $f_\varepsilon=\varepsilon f$ for some $f \in \CI^{2,\alpha}(\pd\M)$. By \eqref{anal_epsilon} there holds $\pd_\varepsilon u_\varepsilon|_{\varepsilon=0}=\mathcal G_{V_1}^D f$. Moreover, since $u_\varepsilon \in \CI^{2,\alpha}(\M)$ and since $\pd_\nu u_\varepsilon\,|_{\pd \M}$ is determined through the map $\Lambda_{V}$, we can simply write 
$$ \pd_\nu \mathcal G_{V_1}^D\,f=\pd_\nu \pd_{\varepsilon}u_\varepsilon|_{\varepsilon=0}=\pd_\varepsilon\,\Lambda_{V}f_\varepsilon\,|_{\varepsilon=0},$$
which shows that $\Lambda^{\text{lin}}_{V_1}f=\pd_\varepsilon \Lambda_{V}(\varepsilon f)\,|_{\varepsilon=0}$.

We can also use this linearization technique to identify interactions for solutions to the linearized equation $\mathcal P_{V_1}u=0$. Indeed, let us consider $\varepsilon \in \C^m$ with $m \geq 2$ and $\{f_k\}_{k=1}^m\subset \CI^{2,\alpha}(\pd \M)$. Since $u_\varepsilon$ solves equation \eqref{pf1} with Dirichlet data $f_\varepsilon$ given by \eqref{lin_data} and since the dependence on $\varepsilon$ is analytic, it follows that given any multi-index $\beta \in \{0,1,\ldots\}^m$ with $|\beta|=\beta_1+\ldots+\beta_m\geq 1$, the function $\p_{\varepsilon}^{\beta} u_\varepsilon|_{\varepsilon=0}$ solves 
$$ \mathcal P_{V_1}(\p_{\varepsilon}^{\beta} u_\varepsilon|_{\varepsilon=0})=F_{\beta}$$ 
with homogeneous Dirichlet boundary conditions, where $F_\beta$ depends on $V_1,\ldots,V_{|\beta|}$ and $\p^{\beta'}_\varepsilon u_\varepsilon|_{\varepsilon=0}$ with $|\beta'|=1,\ldots,|\beta|-1$. Using a simple induction argument it follows that $\p^\beta_\varepsilon u_\varepsilon|_{\varepsilon=0}$ only depends on $V_1,\ldots, V_{|\beta|}$ and $f_1,\ldots,f_m$.

We now consider a particular term in the power series expansion of $u_\varepsilon$ near $\varepsilon=0$ associated to the multi-index $\beta=(1,1,\ldots,1)$ and define
$$L_{f_1,\ldots,f_m}:=-\frac{\pd^{m}}{\pd \varepsilon_1\ldots\pd \varepsilon_m} u_\varepsilon \,|_{\varepsilon=0}.$$
It follows that the function $L_{f_1\ldots f_m}$ satisfies the equation
\bel{eq_lin}
\begin{aligned}
\begin{cases}
\mathcal P_{V_1}\,L_{f_1\ldots f_m}\,=V_m\,\prod_{k=1}^m \mathcal G_{V_1}^D f_k + H_{f_1,\ldots,f_m}, 
&\forall x \in \M
\\
L_{f_1\ldots f_m}= 0
&\forall x \in \pd \M
\end{cases}
    \end{aligned}
\ee
where $H_{f_1,\ldots,f_m}\in \mathcal C^{2,\alpha}(\M)$ is a function that only depends on $V_1,\ldots,V_{m-1}$ and $\p_\varepsilon^{\beta'} u_\varepsilon|_{\varepsilon=0}$ with $|\beta'|=1,\ldots,m-1$. Using the argument in the previous paragraph, it follows that $H_{f_1,\ldots,f_m}$ only depends on $V_1,\ldots,V_{m-1}$ and $f_1,\ldots,f_{m}$. In the particular case $m=2$, we have $H_{f_1,f_2}\equiv 0$. Finally, let us emphasize that since the Dirichlet-to-Neumann map, $\Lambda_V$, determines $\pd_\nu u_\varepsilon\,|_{\pd \M}$, it will also uniquely determine the values $\pd_\nu L_{f_1,\ldots,f_m}\,|_{\pd \M}$.

\subsection{Reduction to the case $c \equiv 1$}
\label{c_reduction}
This subsection is concerned with showing that one can without any loss in generality consider the case where $c(x^0,x')\equiv 1$ on $\M$. Let us define $\hat{\g}=(dx^0)^2+g(x)$ so that $\g=c\hat{\g}$. Using the transformation law of the Laplace-Beltrami operator under conformal rescalings of the metric, we write
\bel{conformal} c^{\frac{n+2}{4}}(-\Delta_{\g}u+V(x,u))=-\Delta_{\hat{\g}}v+ \hat{V}(x,v),\ee
where $v=c^{\frac{n-2}{4}}u$ and $\hat{V}(x,v)=c^{\frac{n+2}{4}}V(x,c^{-\frac{n-2}{4}}v)-(c^{\frac{n-2}{4}}\Delta_{\g}\,c^{-\frac{n-2}{4}})v$. It can be easily checked that conditions (i)$-$(iv) also hold for the function $\hat{V}(x,z)$. Moreover, if $V_1$ is smooth then so is the function $\hat{V}_1$.

Let $r'_0,r'_1>0$ and consider solutions $v \in B^{\alpha}_{r'_1}(\M)$ to equation 
\bel{pf2}
\begin{aligned}
\begin{cases}
-\Delta_{\hat{\g}}v+\hat{V}(x,v)=0, 
&\forall x \in \M
\\
v= f \in B_{r_0'}^{\alpha}(\pd\M)
&\forall x \in \pd \M
\end{cases}
    \end{aligned}
\ee 
It can be easily verified that for $(r_0',r_1')$ small depending on $(r_0,r_1)$ and $\|c\|_{\CI^3(\M)}$, equation~\eqref{pf2} has a unique solution given by $v=c^{\frac{n-2}{4}}u$ where $u$ is the unique solution to equation \eqref{pf1} subject to Dirichlet data $c^{-\frac{n-2}{4}}f$. We can therefore uniquely determine the DN map $\Lambda_{\hat{V}}$ for equation~\eqref{pf2} from the DN map $\Lambda_V$ for equation \eqref{pf1} and henceforth consider the problem of determining $\hat{V}$ from $\Lambda_{\hat{V}}$. Finally note that once uniqueness is proved for $\hat{V}$, we can immediately deduce uniqueness for $V$. Thus, without loss of generality we will make the standing assumption throughout the rest of the paper that $c \equiv 1$.

\subsection{A Carleman estimate}
\label{section_semiclassical}
This section is concerned with providing a right inverse for the following differential operator 
\[
\mathcal L_{\lambda}\cdot := e^{-\lambda x^0}\left(-\Delta_{\g}+V_1\right)(e^{\lambda x^0}\cdot),
\]
where $\lambda \in \R$ and $|\lambda|$ is sufficiently large, with suitable continuity estimates in $H^k(\M)$ norm for any fixed $k \in \N$ (see Proposition~\ref{semiclassical}). 
 
We start by introducing some notation. Choose an arbitrarily small auxiliary extension of the manifold $M$ into a smooth closed manifold $\hat{M}$ without boundary, and smoothly extend the metric $g(x')$ to $\hat{M}$. We also extend $V_1$ smoothly to $\hat{T}=I\times \hat{M}$ so that $V_1 \in \CI^{\infty}_c(\hat{T})$. Here, $I$ is the interval in Definition~\ref{CTA}. Next, for any $m \in \Z$, let $E$ be a bounded linear Sobolev extension operator $E:H^m(\M)\to H^m(\hat{T})$ and denote by $\{\psi_l\}_{l\in\N}$, the set of orthonormal eigenfunctions for the Laplace operator on $(\hat{M},g)$, so that $-\Delta_g \psi_l = \mu_l \psi_l$ with $\{\mu_l\}_{l\in \N}$ denoting the eigenvalues.

We have the following proposition.

\begin{proposition}
\label{semiclassical}
Let $k\in \N$ and suppose that $(\M,\g)$ is a CTA manifold as above and that $V_1 \in \CI^{\infty}(\M)$. Then there exists $\lambda_0>0$, depending on $(\M,\g)$, $V_1$ and $k$, such that for all $|\lambda|>\lambda_0$ with $\lambda^2 \notin \{\mu_n\}_{n\in\N}$,  the equation 
\bel{eq} \mathcal L_{\lambda} r = f, \quad f \in H^{k}(\M),\ee
admits some solution $u \in H^{k}(\M)$, satisfying the estimate
$$\|r\|_{H^m(\M)}\leq C \lambda^{-1} \|f\|_{H^m(\M)} \quad \text{for $m=0,\ldots,k$},$$
where the constant $C>0$ is independent of $\lambda$.
\end{proposition}

Let us remark that in the case $k=0$, this is well-known, see for instance \cite[Proposition 4.4]{DosSantosFerreira2009} and \cite[Theorem 4.1]{Salo_notes}. We will present the proof here, as we need existence results with control on the $H^k(\M)$ norm with $k$ large. The proof here will be based on first extending to the infinite cylinder $\hat{T}$ and then applying Fourier mode analysis with respect to the tranversal manifold $(M,g)$. This is similar to the case $k=0$ as presented in \cite[Chapter 4]{Salo_notes}. Let us remark that it is also possible to use the Carleman estimates for the adjoint operator $\mathcal L_{-\lambda}$ shifted to negative Sobolev spaces (see for example \cite[Lemma 4.3]{DosSantosFerreira2009} and \cite[Section 4]{KSU07}) and a standard duality argument to obtain similar right inverses. However, the estimates obtained using this approach will be in semi-classical norms. The estimates that we obtain by using the Fourier analysis approach are slightly stronger due to the fact that the operator $\mathcal L_\lambda$ has constant coefficients with respect to the $x^0$ variable. 

\begin{proof}[Proof of Proposition~\ref{semiclassical}]
We only provide the proof for the case $V_1 \equiv 0$. For the case that $V_1$ is smoothly supported in $\hat{T}$, the proof here together with the exact arguments as in \cite[Theorem 4.1]{Salo_notes} yields the result. We begin by introducing an operator $S_a$ defined for any non-zero  $a \in \R$ , and any $h \in H^k(\R)$, as follows.
$$ (S_a h)(x)= \mathscr F^{-1}\left(\frac{(\mathscr F h)(\xi)}{i\xi+a}\right),$$
where $\mathscr F$ denotes the Fourier transform on $\R$. Using similar arguments as in \cite[Proposition 4.4]{Salo_notes}, we have that for all $m=0,1,\ldots,k$ and all $\delta>\frac{1}{2}$:
\bel{ode} 
\begin{aligned}
\|S_a h\|_{H^m(\R)}&\leq C a^{-1} \|h\|_{H^m(\R)},\quad \forall |a|>1.\\
\|S_a h\|_{H^m_{-\delta}(\R)}&\leq C \|h\|_{H^m_{\delta}(\R)}\quad \forall a \neq 0,
\end{aligned}
\ee
with $C$ independent of the parameter $a$ and $\|h\|^2_{H^m_\delta(\R)}:=\sum_{j=0}^m \|(1+|x|^2)^{\frac{\delta}{2}}\pd_x^k\,h\|^2_{L^2(\R)}$.

Let $F:=Ef$ with $F$ compactly supported in $\hat{T}$. We begin by writing $F$ in terms of the eigenfunctions of $\hat{M}$ as follows
$$ F(x^0,x')=\sum_{l \in \N} F_l(x^0) \psi_l(x').$$
In \cite[Proposition 4.6]{Salo_notes}, it was proved that the function
$$R(x^0,x'):=\sum_{l \in \N} R_l(x^0) \psi_l(x'),$$
with
$$ R_l=S_{\lambda+\sqrt{\mu_l}}S_{\lambda-\sqrt{\mu_l}}\,F_l$$
solves the equation \eqref{eq} on the larger set $\hat{T}$ and satisfies the estimate
$$\|R\|_{H^2_{-\delta}(\hat{T})}\lesssim \lambda \|F\|_{L_\delta^2(\hat{T})},$$
where we are using the notation
$$ \|\cdot\|^2_{H^m_{\delta}(\hat{T})}:=\sum_{j=0}^m \|(1+|x^0|^2)^{\frac{\delta}{2}}D^{m}\,\cdot\|^2_{L^2(\hat{T})}.$$
Now defining $r:=\mathbb I_{\M} R$, with $\mathbb I_{\M}$ denoting the characteristic function of $\M$, it is clear that $r \in H^2(\M)$ solves equation \eqref{eq} on $\M$. 

Let us proceed to prove the claimed bound in the statement of the proposition. We start by noting that given any $p=0,\ldots,k$ and $m=0,\ldots, p$, there holds
\begin{multline*}
\|\sum_{l \in \N} (\pd_{x^0}^{m} R_l) \mu_l^{\frac{p-m}{2}} \psi_l\|^2_{L^2_{-\delta}(\hat{T})}=\sum_{l \in \N} \mu_l^{p-m}\|S_{-\lambda+\sqrt{\mu_l}}S_{\lambda+\sqrt{\mu_l}}\,\pd^m_{x^0}F_l\|^2_{L_{-\delta}^2(\hat{T})}\\
\lesssim \lambda^{-2}\sum_{l \in \N} \mu_l^{p-m}\|\pd^m_{x^0}F_l\|^2_{L^2(\hat{T})},
\end{multline*}
where in the last step, we have used the bound \eqref{ode}, the fact that $|\lambda - \sqrt{\mu_n}|>0$ and that $F$ is compactly supported in $\hat{T}$. Observing that 
$$\pd^m_{x^0}F_l(x^0)=\int_{\hat{M}}\pd^m_{x^0} F(x^0,x')\psi_l(x')\,dV_{g},$$
together with the fact that $\hat{M}$ is closed, we deduce that
$$\|\sum_{l \in \N} (\pd_{x^0}^{m} R_l) \mu_l^{\frac{p-m}{2}} \psi_l\|_{L_{-\delta}^2(\hat{T})}\lesssim \lambda^{-1}\|\pd^m_{x^0}(-\Delta_{g})^{\frac{p-m}{2}}F\|_{L^2(\hat{T})}.$$
To complete the proof, we write
\begin{multline*}
\|r\|_{H^k(\M)}\leq \|R\|_{H^k(\hat{T})}\lesssim \sum_{p=0}^k\sum_{m=0}^{p}\|\sum_{l \in \N} (\pd_{x^0}^{m} R_l) \mu_l^{\frac{p-m}{2}} \psi_l\|_{L^2(\hat{T})}\\
 \lesssim \lambda^{-1} \|F\|_{H^k(\hat{T})}\lesssim \lambda^{-1} \|f\|_{H^k(\M)}.
\end{multline*}
\end{proof}

\section{The Jacobi weighted ray transform}
\label{section_jacobi}
This section is concerned with the introduction of a geometrical data related to the transversal manifold $(M,g)$ that will appear later in the proof of Theorem~\ref{t1}--\ref{t2}. Before proceeding, let us introduce some notation, following \cite[Section 1.2]{MR2571812}. Given a maximal unit speed geodesic $\gamma(t) \subset M$ with $t \in [\tau_-,\tau_+]$, we define the {\em orthogonal complement}, $\dot{\gamma}(t)^\perp$, at the point $\gamma(t)$ as the set
$$\dot{\gamma}(t)^\perp:=\{v \in T_{\gamma(t)}M\,|\, g(\dot{\gamma}(t),v)=0\}.$$
We also define the $(1,1)$-tensor $\Pi_{\gamma}(t)=\Pi^j_i(t)\,\frac{\pd}{\pd x^j}\otimes\,dx^i$ to be the projection from $T_{\gamma(t)}M$ onto $\dot{\gamma}^\perp(t)$. Finally, we say that a $(1,1)$-tensor $L(t)$ along $\gamma$ is {\em transversal} if $ \Pi_\gamma\, L\, \Pi_\gamma=L$.
Transversal $(1,1)$-tensors can be viewed as linear maps from $\dot{\gamma}^{\perp}(t)$ to itself. Now, given such a tensor $L$, we consider the complex Jacobi equation
\bel{jacobi} \frac{D^2}{dt^2}L(t)-K(t) L(t)=0,\ee
with the initial condition
$$ L(0)=L_0,\quad \dot{L}(0)=L_1.$$
Here, 
$$K= K^i_j \frac{\pd}{\pd x^i}\otimes \,dx^j,\quad K^i_j=g^{ik}R_{kj}$$ 
where $R$ denotes the Ricci tensor. We recall from \cite{MR2571812} that if a complex $(1,1)$-tensor $L$ solves the complex Jacobi equation and $L_0,L_1$ are transversal, then $L(t)$ is transversal along $\gamma(t)$ for all $t \in [\tau_-,\tau_+]$.

For each maximal geodesic $\gamma$ in $M$, we let  $\mathbb Y_{\gamma}$ denote the set of all transversal $(1,1)$-tensors $Y(t)$ that solve equation \eqref{jacobi} subject to the
additional constraint that
\bel{Y cond} 
\begin{aligned}
&\text{$Y(\tau_0)$ is non-degenerate, $\dot{Y}(\tau_0)Y(\tau_0)^{-1}$ is symmetric} \\
&\text{and $\Im(\dot Y(\tau_0)Y(\tau_0)^{-1})>0$ for some $\tau_0\in [\tau_-,\tau_+]$}.
\end{aligned}
\ee
Here, $\Im$ denotes the imaginary part and we recall that $[\tau_-,\tau_+]$ is the interval of definition associated to the maximal unit speed geodesic $\gamma(t)$ in $M$. We now define, for all $Y \in \mathbb Y_{\gamma}$, the Jacobi weighted ray transform of the first and second kind, $\mathscr J^{(1)}_{Y}$ and $\mathscr J^{(2)}_{Y}$, as follows.

\bel{jacobi transform}
\begin{aligned}  
\mathscr J^{(1)}_{Y}\,f:&=\int_{\tau_-}^{\tau_+}f(\gamma(t))\,(\det Y(t))^{-\frac{1}{2}}\,dt\quad\forall\,f \in \CI(M),\\
\mathscr J^{(2)}_{Y}\,f:&=\int_{\tau_-}^{\tau_+} f(\gamma(t))\,|\det Y(t)|^{-1}\,dt \quad \forall\, f \in \CI(M).
\end{aligned}
\ee
The following lemma guarantees that $\mathscr J^{(j)}_Y$, $j=1,2$ is well-defined.
\begin{lemma}
\label{Ywell-posed}
For all $Y \in \mathbb Y_{\gamma}$ we have that 
$Y(t)$ is non-degenerate, $\dot{Y}(t)Y(t)^{-1}$ is symmetric and $\Im(\dot Y(t)Y(t)^{-1})>0$ for all $t \in [\tau_-,\tau_+]$.
\end{lemma}
We refer the reader to \cite[Lemma 2.56]{MR1889089} for the proof of this lemma. In the following subsections, we will study the injectivity of the Jacobi ray transforms of the first and second kind along a single maximal geodesic $\gamma$. Before presenting the main results and their proofs, let us give a heuristic discussion to shed some light on the approach. Recall that the matrices $Y(t)$ solve second order ODEs along the geodesics. Therefore it is possible to choose weights in the Jacobi transform that have a limiting singular behavior at any fixed point $p$. We will see that this singularity will also appear at all points that are conjugate to $p$ on $\gamma$. 

In the case of the Jacobi transform of the second kind, the transform will have a limiting singularity at $p$ and its conjugate points and local information can be obtained in all dimensions under the admissibility assumption (Proposition~\ref{jacobi ray inversion}). In the case of the Jacobi ray transform of the first kind, the proof is more delicate as there is no limiting singularity in the transform. Here, the imaginary part of the transform can be localized to deduce injectivity. We have presented injectivity of the Jacobi transform of the first kind along a single geodesic, only in dimensions two and three and left the higher dimensional cases open (Proposition~\ref{jacobi ray inversion 1}). This will be further discussed in Remark~\ref{rmk_jacobi}. 

\subsection{Inversion of Jacobi weighted ray transform of the first kind}

This subsection is concerned with the following injectivity result.

\begin{proposition}
\label{jacobi ray inversion 1}
Suppose $(M,g)$ is a two or three dimensional compact smooth Riemannian manifold with boundary. Suppose that $\gamma$ is a maximal geodesic in $M$ that contains no conjugate points. Let $f \in \mathcal C(M;\R)$. The following injectivity result holds:
$$ \mathscr J^{(1)}_{Y} f=0, \quad \forall\,Y \in \mathbb Y_{\gamma}\quad \implies \quad f(\gamma(t))=0\quad \forall \,t\in[\tau_-,\tau_+].$$
\end{proposition}

\begin{proof}[Proof of Proposition~\ref{jacobi ray inversion 1}]
For consistency of notation, we will denote the dimension of $M$ by $n-1$ with $n \in \{3,4\}$ throughout this proof.

Let us first consider the case when $n=3$. In this case the Jacobi matrices will simply be complex valued scalar functions along the geodesic, that solve \eqref{jacobi} on the maximal interval $[\tau_-,\tau_+]$. It is in fact possible to consider this equation on a slightly larger interval $[\tau_-',\tau_+']$ after choosing a smooth extension of $K$. We consider a family of solutions $Y^\epsilon:[\tau_-',\tau_+']\to \C$ to \eqref{jacobi} with $\epsilon>0$, subject to the initial data 
\bel{initial data'}Y^\epsilon (\tau_-')= -i\epsilon  \quad \text{and} \quad \dot{Y}^\epsilon(\tau_-')=1.\ee
Observe that $Y^\epsilon \in \mathbb Y_\gamma$ since condition \eqref{Y cond} is satisfied for $\tau_0=\tau_-'$ for all $\epsilon>0$. Note also that 
$$ Y^\epsilon= X- i\epsilon Z$$
where $X$ and $Z$ are real-valued solutions to \eqref{jacobi} subject to $X(\tau_-')=0$, $\dot{X}(\tau_-')=1$ and $Z(\tau_-')=1$, $\dot{Z}(\tau_-')=0$. We also record that since $\gamma$ contains no conjugate points on $[\tau_-,\tau_+]$, it follows that $X(t)$ is strictly positive for $t \in [\tau_-,\tau_+]$. 

Recall that by the hypothesis of the proposition, $$\mathscr J^{(1)}_{Y^{\epsilon}}\,f=0$$ for all $\epsilon>0$. This equation reduces to 
$$\int_{\tau_-}^{\tau_+} f(\gamma(t))X(t)^{-\frac{1}{2}}\,(1-i\epsilon \tilde{X})^{-\frac{1}{2}}\,dt=0,$$
where $\tilde{X}(t)=Z(t)X(t)^{-1}$. Applying the Taylor series approximation of $(1-i\epsilon \tilde{X})^{-\frac{1}{2}}$ near $\epsilon=0$ we deduce that
\bel{id1}\int_{\tau_-}^{\tau_+} f(\gamma(t))X(t)^{-\frac{1}{2}}\tilde{X}(t)^{k}\,dt=0,\quad \text{for $k=0,1,\ldots$}.\ee

We claim that $\tilde{X}(t)$ is strictly decreasing on $[\tau_-,\tau_+]$. To see this, observe that
$$\dot{\tilde{X}}(t)=(\dot{Z}(t)X(t)-\dot{X}(t)Z(t))X(t)^{-2}=W_{Z,X}(t)X(t)^{-2},$$
where $W_{Z,X}(t)$ denotes the Wronskian corresponding to the Jacobi equation \eqref{jacobi} and as such satisfies
$$ W_{Z,X}(t)=W_{Z,X}(\tau_-')=-1,\quad t \in [\tau_-,\tau_+]$$
where we are using the fact that $X,Z$ are scalar functions that solve a second order ODE and therefore their Wronskian is constant. Using this observation, we conclude that
$$ \dot{\tilde{X}}(t)<0, \quad \forall\, t \in [\tau_-,\tau_+],$$
implying that $\tilde X$ is strictly decreasing on $[\tau_-,\tau_+]$. We can consequently define the new variable $\tilde t=\tilde X(t)$ and rewrite \eqref{id1} in terms of $\tilde t$ as
$$ \int_{\tilde{X}(\tau_+)}^{\tilde{X}(\tau_-)} f(\gamma(\tilde{X}^{-1}(\tilde t)))\frac{X(\tilde{X}^{-1}(\tilde t))^{-\frac{1}{2}}}{\tilde{X}'(\tilde{X}^{-1}(\tilde t))}\tilde{t}^{k}\,d\tilde t=0,\quad \text{for $k=0,1,\ldots$}.$$
Finally, using the Stone-Weierstrass theorem, and since $f \in \CI(M)$, we conclude that $f$ must identically vanish along $\gamma$, thus concluding the proof for the case $n=3$. 

Let us now consider the case that $n=4$. We consider an arbitrary point $p \in \gamma$ and assume without loss of generality that $p=\gamma(0)$. Let $\{v_1,v_2\} \subset \dot{\gamma}^{\perp}(0)$ be an orthonormal basis for $\dot{\gamma}(0)^\perp$ and for each $\epsilon>0$, consider the unique transversal $(1,1)$-tensor $Y^{\epsilon}$ solving the Jacobi equation \eqref{jacobi} subject to initial conditions
\bel{initial data 0} 
Y^{\epsilon}_k(0)=-i\epsilon\, v_k\quad \text{and}\quad \dot{Y}_k^{\epsilon}(0)=v_k\quad \text{for $k=1,2$},
\ee
where $Y^{\epsilon}_k$ is the $k^{\text{th}}$ column of the tensor $Y^{\epsilon}$ at the point $p$ with respect to the basis $\{v_1,v_2\}$. Observe that condition \eqref{Y cond} holds at $t=0$ for each $\epsilon>0$, implying that $Y^{\epsilon} \in \mathbb Y_{\gamma}$. By the hypothesis of the proposition, we have \bel{r01}\mathscr J^{(1)}_{Y^{\epsilon}}\,f =0,\quad \forall\, \epsilon>0.\ee

We can write $Y^{\epsilon}=X-i\epsilon Z$ where $X,Z$ solve \eqref{jacobi} subject to the initial conditions
\[
X_k(0)=0,\quad Z_k(0)=v_k \quad \text{and}\quad \dot{X}_k(0)=v_k,\quad\dot{Z}_k(0)=0,\quad \text{for $k=1,2$}.
\]

Our aim now is to study the behavior of $\mathscr J^{(1)}_{Y^{\epsilon}}\,f$ as $\epsilon$ approaches zero. Observe that $Y^\epsilon$ with $\epsilon=0$ is singular at $t=0$. This singularity will allow us to obtain information about the value of $f$ at the point $p$. However, $Y^{\epsilon}$ is analytic with respect to $\epsilon$ and studying the asymptotic behavior of $\mathscr J_{Y^\epsilon}$ itself will not contain any local information about $f$ on $\gamma$. To remedy this issue, we define 
$$\mathscr S_\epsilon f= \mathscr J^{(1)}_{Y^{\epsilon}}\,f - \overline{\mathscr J^{(1)}_{Y^{\epsilon}}\,f},$$
and note that by the hypothesis of the proposition, and since $f$ is real valued we have $\mathscr S_\epsilon f=0$. We now choose a small positive parameter $\zeta$ and assume that $0<\epsilon<\zeta$. In what follows, we will study the limiting behavior of $\mathscr S_\epsilon\, f$ as $\epsilon$ approaches zero while $\zeta$ is fixed.

Writing $\{e_1(t),e_2(t)\}$ to denote the parallel transport of the orthonormal basis $\{v_1,v_2\}$ along $\gamma$, it is easy to see that $Y^{\epsilon}(t)$ can be thought of as a two by two matrix with respect to the basis $\{e_1(t),e_2(t)\}$. Next, using the Taylor series approximation for the matrix $Y^\epsilon$ near $t=0$, we deduce that given any $|t|<\zeta$ and all $\epsilon<\zeta$, there holds
\bel{determinant0}
\begin{aligned}
\begin{cases}
|Y^{\epsilon}_{ij}(t)| \leq C_0t^2
&\text{for $i\neq j$ and $i,j=1,2$}
\\
|Y_{jj}^{\epsilon}(t)-(t-i\epsilon)|\leq C_0t^2,
&\text{for $j=1,2$},
\end{cases}
    \end{aligned}
\ee
where $C_0>0$ is independent of $\zeta$ and $\epsilon$. Here, we are considering the matrix $Y^\epsilon$ with respect to the basis $\{e_1,e_2\}$. Applying these estimates to the expression for $\det Y^{\epsilon}$, and noting that $|t-i\epsilon|>|t|$, we deduce that there exists a constant $C_1>0$ independent of $\epsilon,\zeta$, such that
\[
\det Y^{\epsilon}(t)=(t-i\epsilon)^{2}(1+ r^{\epsilon}(t))\quad \text{with}\quad |r^{\epsilon}(t)| \leq C_1 |t|,
\]
for all $t \in (-\zeta,\zeta)$. Consequently, there exists a constant $C_2>0$ independent of $\epsilon,\zeta$ such that
\bel{bound_00}
\left|(\det Y^{\epsilon})^{-\frac{1}{2}} -\overline{(\det Y^{\epsilon})^{-\frac{1}{2}}}-2i\Im ((t-i\epsilon)^{-1}) \right| \leq C_2,
\ee
for all $t \in (-\zeta,\zeta)$. 

Let us now consider the interval $[\tau_-,\tau_+] \setminus(-\zeta,\zeta)$. First, note that $\det X(0)=0$ and that by Definition~\ref{adm}, no point on $\gamma(t)$ is conjugate to $\gamma(0)$ for $t \neq 0$ (see for example \cite[Section 5.5]{do1992riemannian}). We deduce that
$$ \det X \neq 0\quad \text{on}\quad [\tau_-,\tau_+]\setminus (-\zeta,\zeta).$$
Note also that by applying the point-wise bounds in \eqref{determinant0} for $\epsilon=0$, it follows that,
$$ |\det X(t)-t^2|\leq C t^3,\quad \text{for all $|t|\leq \zeta$},$$
for some $C$ independent of $\zeta$. Together with the fact that $X$ is non-degenerate away from the origin, we conclude that 
$$ \det X(t) >0 \quad \forall \, t  \in [\tau_-,\tau_+]\setminus \{0\}.$$
Using this observation, we write 
$$ \det Y^{\epsilon}=\det X\det(I-i\epsilon ZX^{-1})=\det X - i\epsilon \det X\Tr (ZX^{-1})+\mathcal O(\epsilon^2)$$
where we applied the expansion formula for the characteristic polynomial of matrices in the last step. Since $\det X$ is strictly positive away from the origin, we can conclude that there exists a constant $C_\zeta>0$ only depending on $\zeta$, such that
\bel{bound_20}
\left| (\det Y^{\epsilon})^{-\frac{1}{2}} -\overline{(\det Y^{\epsilon})^{-\frac{1}{2}}}\right| \leq C_\zeta\, \epsilon \quad \text{for $t \in [\tau_-,\tau_+]\setminus (-\zeta,\zeta)$}.
\ee
Let us now analyze the limiting behavior of $\mathscr S_\epsilon \,f$ as $\epsilon$ approaches zero. To this end, we begin by writing
$$ 0=\mathscr S_\epsilon\,f = A^{(1)}_{\zeta,\epsilon}+A^{(2)}_{\zeta,\epsilon}+A^{(3)}_{\zeta,\epsilon},$$ 
where 
\bel{As}
\begin{aligned}
A^{(1)}_{\zeta,\epsilon}&=\int_{-\zeta}^{\zeta}f(\gamma(0))\left((\det Y^{\epsilon})^{-\frac{1}{2}} -\overline{(\det Y^{\epsilon})^{-\frac{1}{2}}}\right)\,dt,\\
A^{(2)}_{\zeta,\epsilon}&=\int_{-\zeta}^{\zeta}(f(\gamma(t))-f(\gamma(0)))\left((\det Y^{\epsilon})^{-\frac{1}{2}} -\overline{(\det Y^{\epsilon})^{-\frac{1}{2}}}\right)\,dt\\
A^{(3)}_{\zeta,\epsilon}&=\int_{[\tau_-,\tau_+]\setminus (-\zeta,\zeta)}f(\gamma(t))\left((\det Y^{\epsilon})^{-\frac{1}{2}} -\overline{(\det Y^{\epsilon})^{-\frac{1}{2}}}\right)\,dt.
\end{aligned}
\ee
For the term $A^{(1)}_{\zeta,\epsilon}$, we use the bound \eqref{bound_00} together with the estimate 
\bel{osc int}
\int_{-\zeta}^{\zeta} \Im\left((t-i\epsilon)^{-1}\right)\,dt=\int_{-\zeta}^{\zeta} |\Im\left((t-i\epsilon)^{-1}\right)|\,dt=\pi+\mathcal O(\epsilon),
\ee
to write
$$ \left|A^{(1)}_{\zeta,\epsilon}-2\pi i\,f(\gamma(0))\right| \leq C_3\,\|f\|_{L^{\infty}}\zeta+C_4(\zeta)\|f\|_{L^{\infty}}\epsilon,$$
for some $C_3>0$ independent of $\zeta$, $\epsilon$ and $C_4(\zeta)>0$ independent of $\epsilon$. Here, we are using the fact that \eqref{bound_00} is uniformly bounded in $\zeta$. The appearance of $\zeta$ in the first term above is due to the length of the interval of integration in $A^{(1)}_{\zeta,\epsilon}$. 

For the term $A^{(2)}_{\zeta,\epsilon}$, we use the bounds \eqref{bound_00} and \eqref{osc int} again to obtain
$$ \left| A^{(2)}_{\zeta,\epsilon}\right| \leq C_5\, (\omega_f(\zeta)+\|f\|_{L^{\infty}}\zeta)+C_6(\zeta)\,\|f\|_{L^{\infty}}\epsilon,$$
for some $C_5$ independent of $\zeta$, $\epsilon$ and $C_6(\zeta)$ independent of $\epsilon$, where $\omega_f$ denotes the modulus of continuity for the function $f$ at the point $\gamma(0)$. Finally, for the term $A^{(3)}_{\zeta,\epsilon}$, we use the bound \eqref{bound_20} to write
$$ \lim_{\epsilon \to 0}A^{(3)}_{\zeta,\epsilon}=0.$$

We now return to $\mathscr S_\epsilon\,f$, letting $\epsilon$ approach zero. Using the last three estimates for $A^{(k)}_{\zeta,\epsilon}$, $k=1,2,3$, we deduce that
$$ \left|f(\gamma(0))\right| \leq (C_3+C_5)\,\|f\|_{L^{\infty}}\zeta+C_5\,\omega_f(\zeta).$$
Finally, letting $\zeta$ converge to zero, it follows that $f(p)=0$. Since $p \in \gamma$ is arbitrary, it follows that $f(\gamma(t))=0$ for all $t \in [\tau_-,\tau_+]$.
\end{proof}

\begin{remark}
\label{rmk_jacobi}
The restriction on the dimension is mainly due to two technical difficulties. One has to do with the parity of the dimension. Namely, when the dimension of $M$ is odd, an equation of the form \eqref{bound_20} will no longer be valid. This is due to the fact that $\det X$ will change sign at the origin and the imaginary part of $\det Y^\epsilon$ will not become small in $\epsilon$ away from the region $(-\zeta,\zeta)$. This is the reason that we have pursued a different approach for proving the proposition when $\dim M=2$. Another general issue with higher dimensions seems to be the fact that the bound \eqref{bound_00} will no longer hold. 
\end{remark}

\subsection{Inversion of Jacobi weighted ray transform of the second kind}

This section is concerned with the proof of the following proposition.
\begin{proposition}
\label{jacobi ray inversion}
Suppose $(M,g)$ is a smooth compact Riemannian manifold with boundary. Let $p \in M$ and $\gamma$ be a maximal geodesic passing through $p$ that contains no conjugate points to $p$. Let $f \in \mathcal C(M;\C)$. The following injectivity result holds:
$$ \mathscr J^{(2)}_{Y} f=0, \quad \quad \forall \,Y \in \mathbb Y_{\gamma}\quad \implies f(p)=0.$$
\end{proposition}

\begin{proof}[Proof of Proposition~\ref{jacobi ray inversion}]
We will again denote the dimension of $M$ by $n-1$ with $n \geq 3$. We consider the unit-speed parametrization $\gamma:[\tau_-,\tau_+] \to M$ with $\gamma(0)=p$. Let $\{v_1,\ldots,v_{n-2}\} \subset  \dot{\gamma}^{\perp}(0)$ be an orthonormal basis and for each $\epsilon>0$ sufficiently small, consider the unique $Y^{\epsilon} \in \mathbb Y_{\gamma}$ subject to
\bel{initial data} Y^\epsilon_k(0)=-i\epsilon v_k,\quad \text{and}\quad \dot{Y}_k^\epsilon(0)=v_k,\ee
where $Y^{\epsilon}_k$ is the $k^{\text{th}}$ column of the tensor $Y^\epsilon$ at the point $p$ with respect to $\{v_k\}$ with $k=1,2,\ldots,n-2$. Writing $\{e_k(t)\}_{k=1}^{n-2}$ to denote the parallel transport of the orthonormal basis $\{v_k\}$ along $\gamma$, it is easy to see that $Y^{\epsilon}(t)$ can be thought of as a $(n-2)\times (n-2)$ matrix with respect to the basis $\{e_k(t)\}$ and that $$Y^{\epsilon}(t)=X(t)-i\epsilon\,Z(t)$$ where
$$ \ddot{X}-KX=0, \quad \ddot{Z}-KZ=0,$$ 
\[
X(0)=0,\quad Z(0)=Id\quad \text{and}\quad \dot{X}(0)=Id \quad \dot{Z}(0)=0.
\]
Here, we have realized the $(1,1)$-tensors $K(t)$, $X(t)$ and $Z(t)$ as matrices with respect to the basis $\{e_k(t)\}_{k=1}^{n-2}$.  As in the proof of Proposition~\ref{jacobi ray inversion 1} we remark that since there are no conjugate points $\gamma(t)$ to $\gamma(0)$ away from the origin $t=0$, it follows that $X(t)$ is non-degenerate on $[\tau_-,\tau_+]$ away from $t=0$.

Let $\zeta>0$ be a small parameter. We assume that $0<\epsilon<\zeta$. In what follows, we will study the asymptotic behavior of $\mathscr J^{(2)}_{Y^\epsilon}\, f$ as $\epsilon$ approaches zero while $\zeta$ is fixed. We start by writing 
\bel{asympt_J} \mathscr J^{(2)}_{Y^{\epsilon}}\,f= \underbrace{\int_{-\zeta}^{\zeta}f(\gamma(t))|\det Y^\epsilon(t)|^{-1}\,dt}_{A_{\zeta,\epsilon}}+\underbrace{\int_{[\tau_-,\tau_+]\setminus (-\zeta,\zeta)}f(\gamma(t))|\det Y^\epsilon(t)|^{-1}\,dt}_{B_{\zeta,\epsilon}}.\ee

First, we analyze the term $A_{\zeta,\epsilon}$. Recall that this corresponds to the small neighborhood $(-\zeta,\zeta)$. The following point-wise estimates are analogous to \eqref{determinant0} and hold on the set $t \in (-\zeta,\zeta)$,
\bel{determinant}
\begin{aligned}
\begin{cases}
|Y^{\epsilon}_{ij}(t)| \leq C_0t^2
&\text{for $i\neq j$ and $i,j=1,2,\ldots,n-2$}
\\
|Y_{jj}^{\epsilon}(t)-(t-i\epsilon)|\leq C_0t^2,
&\text{for $j=1,\ldots,n-2$},
\end{cases}
    \end{aligned}
\ee
\\
where $C_0>0$ is independent of $\zeta$ and $\epsilon$. Applying these estimates to the expression for $\det Y^{\epsilon}$, we deduce that
\[
\left| |\det Y^{\epsilon}|- |t-i\epsilon|^{n-2} \right| \leq C_1 \left(\sum_{j=1}^{n-2} t^{2j}|t-i\epsilon|^{n-2-j}\right),
\]
for some $C_1$ independent of $\epsilon,\zeta$ which can be rewritten as
\[
\left|1- |\det Y^{\epsilon}|\,|t-i\epsilon|^{-(n-2)} \right|\leq C_1 \left(\sum_{j=1}^{n-2} t^{2j}|t-i\epsilon|^{-j}\right)\leq C_1\left(\sum_{j=1}^{n-2} |t|^{j}\right).
\]
We deduce that for $t \in(-\zeta,\zeta)$, there holds
\[
\left|1- |\det Y^{\epsilon}|\,|t-i\epsilon|^{-(n-2)} \right| \leq C_2 \zeta,
\]
for some $C_2>0$ independent of $\zeta$ and $\epsilon$. This latter bound implies that
\bel{rand}
\left|1- |\det Y^{\epsilon}|^{-1}\,|t-i\epsilon|^{(n-2)} \right| \leq C \zeta,
\ee
for all $t\in (-\zeta,\zeta)$ where $C>0$ independent of $\epsilon$ and $\zeta$. 

We now write 
$$ A_{\zeta,\epsilon}= \underbrace{\int_{-\zeta}^{\zeta}  f(\gamma(t))\,|t-i\epsilon|^{-(n-2)}dt}_{I}+\underbrace{\int_{-\zeta}^{\zeta}  f(\gamma(t))\,|t-i\epsilon|^{-(n-2)} (-1+|t-i\epsilon|^{(n-2)}|\det Y^{\epsilon}|^{-1})}_{II}dt $$
and
$$ I = f(\gamma(0))(\int_{-\zeta}^{\zeta}|t-i\epsilon|^{-(n-2)}\,dt)+ \underbrace{\int_{-\zeta}^{\zeta}  \left(f(\gamma(t))-f(\gamma(0))\right)\,|t-i\epsilon|^{-(n-2)}\,dt}_{III}.$$
For the term $II$, we use the bound \eqref{rand} to write
\bel{II_bound} |II| \leq C\, \zeta\left(\int_{-\zeta}^{\zeta} |t-i\epsilon|^{-(n-2)}\,dt\right),\ee
where the constant $C$ is independent of $\epsilon, \zeta$. For $III$, we use continuity of $f$ to write
\bel{III_bound} |III| \leq C \omega_f(\zeta)\left(\int_{-\zeta}^{\zeta} |t-i\epsilon|^{-(n-2)}\,dt\right),\ee
where $C>0$ is independent of $\zeta$, $\epsilon$ and $\omega_f$ is a modulus of continuity for $f$. 

Next, we proceed to give a bound on $B_{\zeta,\epsilon}$ in the expression \eqref{asympt_J}. To this end, recall that $X(t)$ is non-degenerate on away from $t=0$ and therefore $|\det X|$ has a positive lower bound on $[\tau_-,\tau_+]\setminus (-\zeta,\zeta)$ that depends on $\zeta$. Since $Y^{\epsilon}$ converges to $X$ as $\epsilon$ approaches zero, we can write
\bel{B} |B_{\zeta,\epsilon}|<C_\zeta,\ee
for some $C_\zeta$ that only depends on $\zeta$. 

We will now divide the entire expression \eqref{asympt_J} by the normalization factor
\bel{normalization_f}\int_{-\zeta}^{\zeta} |t-i\epsilon|^{-(n-2)}\,dt \ee
and study what happens as $\epsilon$ tends to zero. Let us first observe that
$$ \int_{-\zeta}^{\zeta}  |t-i\epsilon|^{-(n-2)}\,dt=\int_{\tau_-}^{\tau_+}  |t-i\epsilon|^{-(n-2)}\,dt-\int_{[\tau_-,\tau_+]\setminus (-\zeta,\zeta)}  |t-i\epsilon|^{-(n-2)}\,dt$$
and that
\bel{Cvalues}
\begin{aligned}
\begin{cases}
\int_{\tau_-}^{\tau_+} |t-i\epsilon|^{-(n-2)}\geq C_n|\log\epsilon|
&\text{for $n=3$}
\\
\int_{\tau_-}^{\tau_+} |t-i\epsilon|^{-(n-2)}\geq C_n\epsilon^{3-n},
&\text{for $n\geq 4$}.
\end{cases}
    \end{aligned}
\ee
and
\bel{C'values}
\begin{aligned}
\begin{cases}
\int_{[\tau_-,\tau_+]\setminus (-\zeta,\zeta)} |t-i\epsilon|^{-(n-2)}\leq C'_n |\log\zeta|
&\text{for $n=3$}
\\
\int_{[\tau_-,\tau_+]\setminus (-\zeta,\zeta)} |t-i\epsilon|^{-(n-2)}\leq C'_n\zeta^{3-n},
&\text{for $n\geq 4$}.
\end{cases}
    \end{aligned}
\ee
where $C_n, C'_n$ are positive constants that are independent of $\zeta$ and $\epsilon$. Combining \eqref{B}$-$\eqref{C'values} we deduce that
$$ \lim_{\epsilon \to 0}\, |B_{\zeta,\epsilon}|\left(\int_{-\zeta}^{\zeta}  |t-i\epsilon|^{-(n-2)}\,dt\right)^{-1}=0.$$

Thus, by dividing \eqref{asympt_J} with the normalization factor \eqref{normalization_f}, using the bounds \eqref{II_bound}--\eqref{III_bound}, and letting $\epsilon$ converge to zero, we conclude that 
 $$\left|f(\gamma(0))\right| \leq C ( \zeta+\omega_f(\zeta)),$$
for some $C>0$ independent of $\zeta$. Finally, by taking the limit $\zeta \to 0$, the proposition follows.
\end{proof}

\section{Complex Geometric Optics}
\label{cgo_section}
The main aim of this section is to construct a pair of so called complex geometric optics solutions $\mathcal U^{\pm}_{\rho}$, with 
$$\rho=\lambda + i\sigma,\quad \lambda>\lambda_0>0,$$ 
for the equation 
\bel{inf_pde} \mathcal P_{V_1} \,\mathcal U^{\pm}_\rho =0 \quad \text{on $T=I\times M$}.\ee 
Here, we have smoothly extended the known function $V_1$ from $\M$ to the larger set $T=I\times M$ such that $V_1 \in \CI^{\infty}_c(T)$. Recall from Section~\ref{prelim} that we have assumed without loss of generality that $c \equiv 1$. We construct solutions that take the form 
\bel{CGO_ansatz} 
\mathcal U^{\pm}_\rho(x)= e^{\pm\lambda x^0} \left(e^{i\sigma x^0}\mathcal V_\rho^{\pm}(x^0,x')+\mathcal R^{\pm}_{\rho}(x^0,x')\right).
\ee
Here, the functions $\mathcal V^{\pm}_\rho$ are directly related to Gaussian quasi modes for the transversal manifold $(M,g)$ and will be supported near the two dimensional sub-manifold $\R \times \gamma$ with $\gamma$ denoting a maximal non-self-intersecting geodesic in $M$. It should be remarked that the Gaussian quasi mode construction is well-known and is analogous to Gaussian beams for the wave equation (see for example \cite{MR646994,MR2053419,MR716507}). The presentation here follows \cite{MR3562352,MR3198591} with some modifications. The correction term $\mathcal R^{\pm}_\rho$ will asymptotically converge to zero for any fixed non-zero $\sigma \in \R$, as $\lambda\to \infty$ with an arbitrary a priori fixed rate of decay $s$: 
\bel{corr_est}\|\mathcal R^{\pm}_{\rho}\|_{\CI^3(\M)}\lesssim \lambda^{-s},\quad \forall \lambda>\lambda_0>0\quad\text{and}\quad |\sigma|\leq \sigma_0.\ee
These statements will be made precise in Proposition~\ref{cgo_sol}.  
\subsection{Gaussian quasi modes}
\label{section_gauss}

Fix a unit speed maximal non-self-intersecting geodesic $\gamma(t) \in M$ with $t \in [\tau_-,\tau_+]$ and extend it as a geodesic to the larger manifold $\hat{M}$ (see Section~\ref{section_semiclassical}) so that it is defined on an interval $[\tau_-',\tau_+']$ with $\tau_-'<\tau_-$ and $\tau_+<\tau'_+$. Let $q$ be a point on $\gamma \cap (\hat{M}\setminus M)$. We define $\{v_{\alpha}\}_{\alpha=2}^{n-1}\subset T_q\hat{M}$ such that $\{\dot{\gamma}(q),v_2,\ldots,v_{n-1}\}$ forms an orthonormal basis and denote by $\{e_{\alpha}(t)\}$, the parallel transport along $\gamma$ of $\{v_\alpha\}$ to the point $\gamma(t)$. We define 
$$y^0:=x^0 \quad \text{and}\quad y^1:=t,$$ 
and for each $y''=(y^2,\ldots,y^{n-1})$ define the smooth map
$$\mathcal F(y)=\mathcal F(y^0,y^1,y'')=\left(y^0,\exp_{\gamma(y^1)}\left(\sum_{\alpha=2}^{n-1} y^{\alpha}e_{\alpha}(y^1) \right)\right).$$
We use the notation $y=(y^0,y')$ with $y'=(y^1,y'')$ and recall the following lemma (see \cite[Lemma 3.5]{MR3562352}).
\begin{lemma}[Fermi coordinates]
\label{fermi}
Given any sub-interval $[\tau_-'',\tau_+'']$ of $(\tau_-',\tau_+')$ containing $(\tau_-,\tau_+)$, the coordinate system above is a smooth diffeomorphism in a neighborhood $U$ of $I\times\gamma([\tau_-'',\tau_+''])$, and the following statements hold.
\begin{itemize}
\item[(i)]{$\mathcal F^{-1}(U)=I\times (\tau_-',\tau_+')\times B(0,\delta')$, where $B(0,\delta')$ is the ball of radius $\delta'$ centered at the origin in $\R^{n-2}$.}
\item[(ii)]{$\mathcal F^{-1}(y^0,\gamma(y^1))=(y^0,y^1,\underbrace{0,\ldots,0}_{\text{$n-2$ times}})$ for all $y^1\in (\tau_-',\tau_+')$.}
\end{itemize}  
Moreover, $\g(y^0,y')= (dy^0)^2+g(y')$ and $g_{jk}(y^1,0)=\delta_{jk}$, $\frac{\pd g_{jk}}{\pd y^i}(y^1,0)=0$ for $1\leq i,j,k\leq n-1$. 
\end{lemma} 
Let us now return to the task of constructing solutions of the form \eqref{CGO_ansatz}. Let
\bel{tubular}
\mathcal N_{\gamma}=\{y\in \hat{T}\,|\, y^1 \in [\tau_-'',\tau_+''], |y''|<\delta\},
\ee
for some $0<\delta<\delta'$. This is the neighborhood where the Gaussian quasi modes $\mathcal V_\rho$ will be compactly supported. We make the ansatz
\bel{quasi_0}\mathcal V^+_\rho(y^0,y') = e^{i\rho\Theta(y')}a^+_\rho(y)\quad\text{and}\quad \mathcal V^-_\rho(y^0,y') = e^{-i\bar{\rho}\bar{\Theta}(y')}a^-_\rho(y)\ee
The functions $\Theta$, $a^{\pm}_\rho$ are called the phase and amplitude functions respectively. We observe that
\bel{conj}
\begin{aligned}
\mathcal P_{V_1}\, (e^{\rho x^0}\mathcal V^+_\rho)&=e^{\rho (y^0+i\Theta(y'))}\left( -\rho^2 (\mathcal S\Theta)a^+_\rho -\rho \mathcal T^{+} a^{+}_\rho +\mathcal P_{V_1}a^{+}_\rho \right),\\
\mathcal P_{V_1}\, (e^{-\bar{\rho} x^0}\mathcal V^-_\rho)&=e^{-\bar{\rho} (y^0+i\bar\Theta(y'))}\left( -\bar{\rho}^2 (\mathcal S\bar{\Theta})a^-_\rho +\bar{\rho} \mathcal T^{-} a^{-}_\rho +\mathcal P_{V_1}a^{-}_\rho \right),
\end{aligned}
\ee
where
\bel{eikonal}
\mathcal S \Theta := 1- \langle d\Theta,d\Theta\rangle_g.
\ee
and
\bel{transport}
\begin{aligned}
\mathcal T^+ a^+_\rho :&= 2\pd_{y^0} a^+_\rho + 2i\langle d\Theta,da^+_{\rho}\rangle_g + i(\Delta_{g}\Theta)a^+_{\rho}\\
\mathcal T^- a^-_\rho :&= 2\pd_{y^0} a^-_\rho +2i\langle d\bar{\Theta},da^-_{\rho}\rangle_g +i(\Delta_{g}\bar{\Theta})a^-_{\rho}.
\end{aligned}
\ee
Here, we would like to apply the WKB method with respect to the parameter $\rho$ in a neighborhood of $I\times\gamma \subset \hat{\M}$. More specifically, we start by constructing $\Theta(y')$ such that the function $\mathcal S\Theta(y')$ vanishes up to $N^{\text{th}}$ order on the geodesic $\gamma$, that is to say 
\bel{eikonal} \left(\frac{\pd^{\beta}\mathcal S \Theta}{\pd y'^{\beta}}\right)(y^1,0)=0,\quad \forall y^1 \in (\tau_-'',\tau_+''),\ee
for all multi indices $\beta \in \{0,1,\ldots\}^{n-1}$ with $|\beta|\leq N$ . We make the following ansatz,
\bel{phase_0}
\Theta(y^1,y'')=\sum_{k=0}^N \Theta_k(y^1,y''),
\ee
where $\Theta_k(y^1,y'')$ is a homogeneous polynomial of degree $k$ in the transversal variables $y''$. Following \cite[Section 3]{MR3562352}, we can choose
\bel{phase_1}
\Theta_0(y^1,y'')=y^1, \quad \Theta_1(y^1,y'')=0,\quad \Theta_2(y^1,y'')=\sum_{i,j=2}^n \frac{1}{2}H_{ij}(y^1)y^i y^j,
\ee 
where $H:=(H_{ij})_{i,j=2}^n$ solves the following Riccati equation,
\bel{riccati}
\dot{H}(t)+H(t)^2+D(t)=0, \quad t \in (\tau_-'',\tau_+'')\quad H(\tau_0)=H_0.
\ee
Here, $\tau_0 \in (\tau_-'',\tau_+'')$, $D=\frac{1}{2}(\pd^2_{ij}g^{11}|_{\gamma})_{i,j=2}^n$ and $H_0$ is any symmetric matrix with $\Im H_0>0$. The subsequent terms $\Theta_k(y')$ can be constructed by solving linear systems of ODEs and prescribing initial values at the point $t=\tau_0$ and we refer the reader to \cite{MR3562352} for the details. 

Let us now analyze the Riccati equation further as this term dictates the Gaussian type decay away from the two dimensional sub-manifold $\R\times \gamma$. Applying \cite[Lemma 2.56]{MR1889089}), we deduce that given any symmetric $H_0$ with $\Im H_0>0$, there exists a unique solution to equation \eqref{riccati}. Moreover,
\bel{im}\Im(H(t))>0 \quad \forall t \in [\tau_-'',\tau_+'']\ee 
and there holds
\bel{stationary_ph_0}\det(\Im(H(t))\cdot|\det Y(t)|^2=c,\ee
where $c>0$ is a constant independent of $t$ and $Y=(Y_{ij})_{i,j=2}^{n}$ is the unique solution to the second order ODE
\bel{Y}
\ddot{Y}(t)+D(t)Y(t)=0, \quad Y(\tau_0)=Y_0,\quad \dot{Y}(\tau_0)=Y_1.
\ee
with $Y_0$ any non-degenerate matrix and $Y_1=H_0Y_0$. We note that $H(t)$ and $Y(t)$ are related through the expression
$$ H(t)=\dot{Y}(t) Y^{-1}(t).$$
Applying arguments analogous to \cite[Section 3.5]{2019arXiv190104211F} one can show that the matrix $D(t)$ is equal to the (1,1)-Riccci tensor $K$ defined in Section~\ref{section_jacobi}. Therefore, the matrix $Y(t)$ solving \eqref{Y} above coincides with $Y$ in Section~\ref{section_jacobi}, namely that it solves the complex Jacobi equation \eqref{jacobi} (see also \cite{MR2571812}) and satisfies the condition \eqref{Y cond}.

Next, we consider the construction of the amplitude functions $a^{\pm}_\rho$. We write
\bel{amplitude}
\begin{aligned}
a^{+}_\rho(y^0,y')&=\left(v_0(y')+\rho^{-1}v^+_1(y^0,y')+\ldots+\rho^{-N}v^+_N(y^0,y')\right) \chi(\frac{|y''|}{\delta}),\\
a^{-}_\rho(y^0,y')&=\left(\bar{v}_0(y')+\bar{\rho}^{-1}v^-_1(y^0,y')+\ldots+\bar{\rho}^{-N}v^-_N(y^0,y')\right) \chi(\frac{|y''|}{\delta}),
\end{aligned}
\ee
where $\delta$ is as in the definition of $\mathcal N_\gamma$ and $\chi:\R \to \R$ is a smooth non-negative function with $\chi=0$ for $|t|>1$ and $\chi=1$ for $|t|<\frac{1}{2}$. We require that
\bel{transport_principal} \left(\frac{\pd^{\beta}\mathcal T^{+} v_0}{\pd y'^{\beta}}\right)(y^0,y^1,0)=0,\quad \forall (y^0,y^1) \in I\times (\tau_-'',\tau_+''),\ee
and that 
\bel{transport_subprincipal}\left(\frac{\pd^{\beta}(\mp\mathcal T^{\pm} v^{\pm}_k + \mathcal P_{V_1} v^{\pm}_{k-1})}{\pd y'^{\beta}}\right)(y^0,y^1,0)=0,\quad \forall (y^0,y^1) \in I\times (\tau_-'',\tau_+'')\ee
for $k=1,\ldots,N$ and all multi indices $\beta\in \{0,1,\ldots\}^n$ with $|\beta|\leq N$. 
The study of equation \eqref{transport_principal} is presented in \cite[Section 3]{MR3562352}. There, it is showed that if we write
\bel{amplitude_principal_0} v_0(y')=\sum_{j=0}^N v_{0j}(y^1,y''), \ee
with $v_{0j}$ denoting a homogeneous polynomial of degree $j$ in $y''$, then one can take
\bel{amplitude_principal_1}
v_{00}(t)=(\det Y(t))^{-\frac{1}{2}},
\ee
and that the subsequent terms $v_{0j}(t,y'')$ with $j=1,\ldots,N$ can be uniquely determined by solving first order ODEs along the geodesic $\gamma$ subject to some prescired initial conditions at the point $\gamma(\tau_-'')$.

Let us now study equation~\eqref{transport_subprincipal}. Here, we deviate from \cite{MR3562352} due to the presence of $y^0$-dependence in $v_k$, $k\geq 1$. This comes from the fact that we consider $\mathcal P_{V_1}$ with a $y^0$-dependent $V_1$, whereas the case $V_1=0$ is considered in \cite{MR3562352}. Proceeding analogously to the study of \eqref{transport_principal}, we write
$$v^{\pm}_k(y^0,y')=\sum_{j=0}^N v^{\pm}_{kj}(y^0,y^1,y''),$$
where $v^{\pm}_{kj}$ is a homogeneous polynomial in the $y''$ variables of degree $j$. Using the definition of $\mathcal T^{\pm}$, the form of the metric $g(y')$ near $y''=0$, \eqref{transport_subprincipal} reduces to
$$\left(\pd_{y^0} + i \pd_{y^1}\right)(\det Y(y^1)^{\frac{1}{2}}v^{\pm}_{kj}(y))=Q^{\pm}_{kj}(y),\quad (y^0,y^1) \in I \times (-\tau_-'',\tau_+''),\quad j=0,\ldots,N$$
where the functions $Q^{\pm}_{kj}(y^0,y^1,y'')$ is a homogeneous polynomial in the variables $y''$ of degree $j$ only depending on $V_1$ and the preceding terms in the expansion of the amplitude functions. 

To solve for the functions $v^{\pm}_{kj}$ we can proceed with an iterative process by solving at each step, an equation of the form 
$$(\pd_{y^0} + i \pd_{y^1})r=F\quad \quad \forall (y^0,y^1)\in I\times (-\tau_-'',\tau_+'').$$
To solve such an equation, we simply extend $F(y^0,y^1)$ smoothly to $\R^2$ in such a way that $F \in \CI^{\infty}_c(\R^2)$. Let
$$\Gamma(y^0,y^1):=\frac{1}{2\pi i}\,(y^0+iy^1)^{-1},$$
and pick
\bel{transport_fourier}  
r=\mathbb I_{I\times (-\tau_-'',\tau_+'')}\,(\Gamma * F),
\ee
where $\mathbb I_{A}$ is the characteristic function of the set $A$. Note that $r \in \CI^{\infty}(I\times (-\tau_-'',\tau_+''))$. By using this method, we can iteratively determine the coefficients $v^{\pm}_{kj}$ and thus complete the construction of the amplitude functions $a^{\pm}_\rho \in \CI^{\infty}(T)$. 

We have completed the task of constructing the Gaussian quasi modes $\mathcal V^{\pm}_\rho$. Let us point out that for the phase function $\Theta(y')$, we have prescribed the initial conditions for all the ODEs at the point $\tau_0$. This will be later exploited in Section~\ref{section_uniqueness}. To summarize the construction, we state the following key lemma.

\begin{lemma}
\label{quasi_est}
Let $\rho=\lambda+i\sigma$ and let $\mathcal V^{\pm}_\rho \in \CI^{\infty}(\M)$ be constructed as above. Then, for any $|\sigma|<\sigma_0$ and all $\lambda>\lambda_0$ with $\lambda^2 \notin \{\mu_n\}_{n\in\N}$, we have the estimates
\[
\|\mathcal V^{\pm}_\rho\|_{L^p(\M)}\lesssim \lambda^{-\frac{n-2}{2p}}, \quad \|\mathcal V^{\pm}_\rho\|_{\CI^{k}(\M)}\lesssim \lambda^k,\quad \text{for all $p \geq 1$ and $k=0,1,\ldots$}
\]
and
\[
\|\mathcal L_{\pm\lambda}(e^{i\sigma x^0}\mathcal V^{\pm}_\rho)\|_{H^k(\M)}\lesssim \lambda^{2+k-\frac{N}{2}-\frac{n-2}{4}},
\]
where $\mathcal L_\lambda$ and $\{\mu_n\}_{n \in \N}$ are as defined in Section~\ref{section_semiclassical}.
\end{lemma}

\begin{proof}
We will show how to prove these bounds for $\mathcal V_\rho^+$. The bounds for $\mathcal V_\rho^-$ will then follow analogously. First, observe that using equations \eqref{phase_0}--\eqref{phase_1} together with \eqref{im}, we have
\bel{phaseboundquad} |e^{i\rho\Theta(y')}|\lesssim e^{-C_0\lambda|y''|^2},\quad \forall y' \in \mathcal N_\gamma,\ee
where $\mathcal N_\gamma$ is as defined in \eqref{tubular}. Using the Fermi coordinate system together with expressions \eqref{CGO_ansatz} and \eqref{amplitude}, it follows that
$$\|\mathcal V_\rho^+\|_{L^p(\M)}^p \lesssim \int_{I\times \mathcal N_\gamma} e^{-C_0p\lambda|y''|^2}\,dx^0\,dy^1\,dy''\lesssim \lambda^{-\frac{n-2}{2}}$$
thus proving the first claim. For the second claim, we observe that
$$ \|\mathcal V_\rho^+\|_{\CI^k(\M)} \lesssim \lambda^k e^{-C_0\lambda|y''|^2} \lesssim \lambda^k.$$
We now derive the bound for $\mathcal L_{\lambda}(e^{i\sigma x^0}\mathcal V^+_\rho)$. Let us first use the fact that equations \eqref{eikonal}, \eqref{transport_principal} and \eqref{transport_subprincipal} are satisfied together with \eqref{amplitude} to obtain the point-wise bounds on the set $I\times \mathcal N_\gamma$,
\[
\begin{aligned}
\left|\p^l_y\left(\rho^2 (\mathcal S\Theta)a^{+}_\rho-\rho \mathcal T^+ a^+_\rho + \mathcal P_{V_1} a^+_\rho\right)\right|&\lesssim \lambda^2\,|y''|^{N-|l|}+ \lambda\, |y''|^{N-|l|}+\lambda^{-N},\\
\left|\p_y^le^{i\rho\Theta}\right|&\lesssim \lambda^{|l|}\, e^{-C_0\lambda|y''|^2}
\end{aligned}
\]
for multi-indices $l$ with $|l|=0,1,\ldots,k$, where we are using the notation $\p_y^l$ to stand for derivatives of order $l$ with respect to the Fermi coordinates. Next, we recall from \eqref{conj} that
$$\mathcal L_\lambda (e^{i\sigma x^0}\mathcal V^+_\rho) = e^{i\rho\Theta}e^{i\sigma x^0}\left( -\rho^2 (\mathcal S\Theta)a^+_\rho -\rho \mathcal T^+ a^+_\rho +\mathcal P_{V_1}a^+_\rho   \right).$$
Using the previous point-wise bounds we write
\[
\begin{aligned} 
&\|\mathcal L_{\pm\lambda}(e^{i\sigma x^0}\mathcal V^{\pm}_\rho)\|^2_{H^k(\M)}\lesssim \\
&\sum_{|l|=0}^k \int_{I\times \mathcal N_\gamma} \left(\lambda^2\,|y''|^{N-|l|}+ \lambda\, |y''|^{N-|l|}+\lambda^{-N}\right)^2\lambda^{2(k-|l|)}e^{-C_0\lambda|y''|^2}\,dx^0\,dy^1\,dy''\\
&\lesssim \lambda^4 \lambda^{2k} \lambda^{-N-\frac{n-2}{2}}.
\end{aligned}
\]
\end{proof}

\subsection{The remainder term}
\label{section_boundary}
In this section we complete the construction of the complex geometric optic solutions to \eqref{inf_pde} of the form \eqref{CGO_ansatz}. More specifically, we will determine the asymptotically small correction terms $\mathcal R^{\pm}_{\rho}$. 
\begin{proposition}
\label{cgo_sol}
\label{corr_lem}
Let $s \in \N$, $\rho=\lambda+i\sigma$ with $|\sigma|\leq \sigma_0$, $N=13+\frac{n}{2}+2s$ and consider the functions $\mathcal V^{\pm}_\rho$ as above.  There exists solutions $\mathcal U^{\pm}_\rho \in \CI^{3}(\M)$ to equation \eqref{inf_pde} of the form \eqref{CGO_ansatz} satisfying the 
following estimate,
$$\|\mathcal R^{\pm}_\rho\|_{\CI^{3}(\M)} \lesssim \lambda^{-s},$$
for all $\lambda>\lambda_0$, $\lambda^2 \notin \{\mu_l\}_{l\in\N}$.  
\end{proposition}
\begin{proof}
Let us start by noting that the correction term $\mathcal R^{\pm}_\rho$ satisfies the equation
$$ \mathcal L_{\lambda} \mathcal R^{+}_\rho=-\mathcal L_{\lambda}\left(e^{i \sigma x^0} \mathcal V^+_\rho\right),\quad \mathcal L_{-\lambda} \mathcal R^{-}_\rho=-\mathcal L_{-\lambda}\left(e^{i \sigma x^0} \mathcal V^{-}_\rho\right).$$
Now, combining the bounds given in Lemma~\ref{quasi_est} together with Proposition~\ref{semiclassical} we observe that there exists a solution $\mathcal R_\rho^{\pm}$ such that
$$\|\mathcal R_\rho^{\pm}\|_{H^k(\M)} \lesssim \lambda^{2+k-\frac{N}{2}-\frac{n-2}{4}}.$$
Now pick $k=\frac{n}{2}+4$. Using the Sobolev embedding $\mathcal C^{3}(\M) \subset H^{\frac{n}{2}+4}(\M)$, we obtain that
$$\|\mathcal R^{\pm}_\rho\|_{\CI^3(\M)}\lesssim \|\mathcal R^{\pm}_\rho\|_{H^k(\M)} \lesssim \lambda^{-s}.$$ 
\end{proof}

\section{Proof of Theorem 1}
\label{section_uniqueness}
This section is concerned with the proof of Theorem~\ref{t1}. The proof will be built on an induction argument based on $m$, where $m$ is the order of the linearization method discussed in Section~\ref{section_lin}. As the first step of induction and also to shed some light on the methodology, we start with a proposition. 

\begin{proposition}
\label{V_3}
Let the assumptions of Theorem~\ref{t1} hold. Then the DN map $\Lambda_V$ uniquely determines the function $V_3(x)$.
\end{proposition}

\begin{proof}
We start by choosing a point $p \in \mathscr T$ and choose $\gamma$ to be an admissible geodesic passing through $p$ in the sense of Definition~\ref{adm}. Let $\rho=\lambda+i\sigma$ with $\sigma$ fixed and construct a family of complex geometric optic solutions $\mathcal U_\rho^{\pm}$ (see \eqref{CGO_ansatz}) with $\lambda>\lambda_0$ and a decay rate for the correction terms $\mathcal R_\rho^\pm$ given by an integer $s\geq \frac{n}{8}$ (see Proposition~\ref{cgo_sol}). Here, we have assigned the initial values for the ODEs that govern the phase function $\Theta(y')$ at the point $p$. Let us now define 
$$f^{\pm}_\rho=\mathcal U_\rho^{\pm}|_{\pd \M}.$$
Recall that $\mathcal U^{\pm}_\rho \in \CI^{3}(\M)$ are defined by \eqref{CGO_ansatz} and that $\mathcal U^{\pm}_\rho|_{\pd \M}$ is explicitly known since $V_1$ is assumed to be known. By using the definition of $f_{\rho}^{\pm}$ 
together with the arguments in Section~\ref{section_lin} we deduce that we know $\pd_\nu L_{f_\rho^+,f_\rho^+,f^-_\rho}|_{\pd \M}$ for $L_{f_\rho^+,f_\rho^+,f^-_\rho}$ solving the equation
\bel{eq_lin_3}
\begin{aligned}
\begin{cases}
\mathcal P_{V_1}\,L_{f_\rho^+,f_\rho^+,f^-_\rho}\,=V_3\, (\mathcal U_\rho^+)^2\mathcal U_\rho^-+H_{f_\rho^+,f_\rho^+,f^-_\rho}, 
&\forall x \in \M
\\
L_{f_\rho^+,f_\rho^+,f_\rho^-}= 0
&\forall x \in \pd \M
\end{cases}
    \end{aligned}
\ee
We note that $H_{f_\rho^+,f_\rho^+,f^-_\rho}$ is known as it only depends on $V_1$ and $V_2$. Let $d\sigma_{\g}$ denote the volume form on $\pd \M$. Applying Green's identity together with \eqref{inf_pde} and \eqref{eq_lin_3}, we write
\bel{}
\begin{aligned}
&-\int_{\pd \M} f_{\rho}^{-}\, (\pd_{\nu}L_{f_\rho^+,f_\rho^+,f_\rho^-})\,d\sigma_{\g}+\int_{\pd \M} (\pd_{\nu}\mathcal U_{\rho}^{-})\, L_{f_\rho^+,f_\rho^+,f^-_\rho}\,d\sigma_\g \\
&=\int_{\M} \mathcal U_{\rho}^{-} (\mathcal P_{V_1} L_{f_\rho^+,f_\rho^+,f_\rho^-})\,dV_{\g} - \int_{\M} (\mathcal P_{V_1} \mathcal U_{\rho}^{-}) L_{f_\rho^+,f_\rho^+,f^-_\rho}\,dV_{\g}\\
&=\int_{\M}V_3\,(\mathcal U_{\rho}^{-}\mathcal U_\rho^{+})^2\,dV_{\g}+\int_{\M}H_{f_\rho^+,f_\rho^+,f^-_\rho}\mathcal U_\rho^-\,dV_{\g}. 
\end{aligned}
\ee
Thus we can conclude that for all $\lambda>\lambda_0$ with $\lambda^2 \notin \{\mu_l\}_{l \in \N}$, the knowledge of the DN map $\Lambda_V$ uniquely determines the expression 
\begin{multline*}
S=\lim_{\lambda\to \infty} \lambda^{\frac{n-2}{2}} \int_{\M}V_3\,(\mathcal U_{\rho}^{-} \mathcal U_\rho^{+})^2\,dV_{\g}\\
= \lim_{\lambda \to \infty} \lambda^{\frac{n-2}{2}} \int_{\M} V_3(x^0,x') e^{4i\sigma x^0}(\mathcal V_\rho^++\mathcal R^+_\rho)^2(\mathcal V_\rho^-+\mathcal R^-_\rho)^2\,dx^0\,dV_g.
\end{multline*}
Applying Proposition~\ref{cgo_sol} and Lemma~\ref{quasi_est} we have the bounds
\[
\begin{aligned}
\int_{\M} |V_3||\mathcal R_\rho^+|^2|\mathcal R_{\rho}^-|^2\,dV_{\g} &\lesssim |\lambda|^{-4s} \lesssim|\lambda|^{-\frac{n-2}{2}}\lambda^{-1},\\
\int_{\M}|V_3||\mathcal V_\rho^{\pm}|^k|\mathcal R_{\rho}^{\pm}|^j\, dV_{\g} &\lesssim  \lambda^{-\frac{n-2}{2}}\lambda^{-js}.
\end{aligned}
\]
We also recall the point-wise bound \eqref{phaseboundquad} to obtain that 
$$ \lambda^{\frac{n-2}{2}}\int_{I\times N_\gamma} \lambda^{-1}e^{-4\lambda \Im \Theta} \,dx^0\,dt\,dy''\lesssim \lambda^{-1}.$$
Using this bound together with the latter two bounds, the expression for $S$ simplifies to
$$ S = \lim_{\lambda\to \infty}\lambda^{\frac{n-2}{2}}\int_I\int_{\mathcal N_\gamma} V_3(x^0,t,y'') e^{4i\sigma x^0}e^{-4\lambda\Im \Theta}e^{-4\sigma\Re \Theta} |v_{0}(t,y'')|^4\,dx^0\,dV_{g}$$
where we have extended $V_3$ to all of $I \times M$ through setting it to zero outside $\M$. Next, and by using the estimate
\[
\lambda^{\frac{n-2}{2}}\int_{I\times N_\gamma} |y''| e^{-4\lambda \Im \Theta} \,dx^0\,dt\,dy'' \lesssim \lambda^{-\frac{1}{2}}
\]
we can further simplify $S$ to obtain
$$ S=\lim_{\lambda\to \infty}\lambda^{\frac{n-2}{2}}\int_{I\times\mathcal N_\gamma} V_3(x^0,t,0) e^{4i\sigma x^0}e^{-4\sigma t} e^{-4\lambda\Im \Theta}|v_{00}(t)|^4\,dx^0\,dt\,dy''.$$
Finally, applying stationary phase together with equations \eqref{stationary_ph_0} and \eqref{amplitude_principal_1}, we conclude that the Dirichlet-to-Neumann map $\Lambda_V$ uniquely determines the expression
\bel{S0} \int_{\tau_-}^{\tau_+} e^{\xi t}\mathscr F V_3(\xi,\gamma(t)) \,|\det Y(t)|^{-1}\,dt,\quad \text{for all $|\xi|<\frac{\sigma_0}{4}$},\ee
with $\mathscr F V_3$ denoting the Fourier transform of $V_3$ with respect to $x^0$ variable. Here we have extended the function $V_3$ to the set $\R \times M$  by setting it to zero outside of $\M$. We can summarize the analysis thus far as follows. For each point $p \in \mathscr T$ and each admissible geodesic $\gamma$ passing through $p$, we have shown that the Dirichlet-to-Neumann map $\Lambda_V$ uniquely determines the expression \eqref{S0}. Now, recall from \eqref{Y} that the $(n-2)\times (n-2)$ matrix $Y$ solves the complex Jacobi equation \eqref{jacobi} along $\gamma$ subject to the initial condition
$$Y(\tau_0)=Y_0, \quad \text{and}\quad \dot Y(\tau_0)=Y_1=H_0Y_0,$$ 
where $H_0$ is symmetric and satisfies $\Im H_0>0$ and $Y_0$ is non-degenerate. This implies that $Y \in \mathbb Y_\gamma$.

Returning to \eqref{S0} and applying Proposition~\ref{jacobi ray inversion} together with the fact that $\sigma_0$ is arbitrary, we conclude that 
$$ \mathscr F V_3(\xi,p)=0,$$
for all $\xi \in \R$ and all $p \in \mathscr T$. The theorem now follows using the density of $\mathscr T$ in $M$ and applying the inverse Fourier transform.
\end{proof}

Before presenting the proof for the more general case $m \geq 4$, we need a lemma.
\begin{lemma}
\label{nonvanishing}
Given any point $p \in M$, there exists a smooth solution $W_{p}$ to the equation 
$$\mathcal P_{V_1} W_{p}=0$$ 
with the additional property that $W_{p}(p)\neq 0$.
\end{lemma}
\begin{proof}
Let us consider $G(x;y)$ to denote the Dirichlet Green's function associated to the operator $\mathcal P_{V_1}$ on $\M$. For each fixed $y \in \M$, $G(x;y)$ is the unique distributional solution to 
\bel{greens1}
\begin{aligned}
\begin{cases}
\mathcal P_{V_1}G(x;y)=\delta(x;y), 
&\forall x \in \M
\\
G(x;y) = 0 
&\forall x \in \pd \M
\end{cases}
    \end{aligned}
\ee
For each fixed $y$, $G(x;y)$ is smooth away from $y$. Moreover, since $\delta(x;y) \in H^{-\frac{n}{2}-\epsilon}(\M)$ for any $\epsilon>0$ and any fixed $y\in \M$, it follows from elliptic regularity that $G(x;y) \in H^{2-\frac{n}{2}-\epsilon}(\M)$ for any $\epsilon>0$. Let $h \in \CI^{\infty}(\pd \M)$ be arbitrary and choose $w$ such that $\mathcal P_{V_1}w=0$ subject to $w|_{\pd \M}=h$. Applying Green's identity we have
$$ w(x)=\int_{\M}\mathcal P_{V_1} G(x;y) w(y)\,dV_\g-\int_{\M} G(x;y) \mathcal P_{V_1}w(y)\,dV_\g = -\int_{\pd\M} \pd_\nu G(x;y) h(y)\,d\sigma_\g$$
If the claim fails to hold, that is $w(p)=0$ for all $h \in \CI^{\infty}(\pd\M)$, then clearly we must have $\pd_\nu G(p;y)=0$ for all $y \in \pd \M$. Since $G(p;y)$ also vanishes there, by the elliptic unique continuation principle, that $G(x;p)$ must vanish away from $p$ implying that it is supported at the point $\{p\}$. Subsequently, it must be a linear combination of $\delta(x;p)$ and its derivatives at the point $\{p\}$. But as already mentioned in the beginning of the proof $G(x;p)$ is smoother than $\delta(x;p)$, and whence it must vanish everywhere which is a contradiction with \eqref{greens1}. 
\end{proof}
\begin{proof}[Proof of Theorem 1]
We use an induction argument on $m \geq 4$ to show that $V_m$ can be uniquely determined from the DN map $\Lambda_V$. We have already proved this for $m=3$. Now let $m>3$ and assume that $V_j$ has been determined for all $j<m$. We consider a point $p \in \mathscr T$ with an admissible geodesic $\gamma$ passing through $p$. Similar to the proof of Propostion~\ref{V_3}, we pick $s\geq\frac{n}{8}$ and consider the solutions $\mathcal U^{\pm}_\rho$ and define $f_\rho^{\pm}$ analogously. We also choose the function $W_{p}$ as in Lemma~\ref{nonvanishing} and let $h$ be its trace on the boundary $\pd \M$.

Following Section~\ref{section_lin} we define
$$L_m:=L_{f_\rho^{+},f_{\rho}^{+},f_{\rho}^{-},\underbrace{h,\ldots,h}_{\text{$m-3$ times}}}\quad \text{and} \quad H_m:=H_{f_\rho^{+},f_{\rho}^{+},f_{\rho}^{-},\underbrace{h,\ldots,h}_{\text{$m-3$ times}}}.$$
Applying \eqref{eq_lin} we note that $L_m$ solves the equation
\bel{eq_lin_2}
\begin{aligned}
\begin{cases}
\mathcal P_{V_1}\,L_m\,=\tilde{V}_m\,(\mathcal U_\rho^+)^2(\mathcal U^{-}_\rho)+H_m, 
&\forall x \in \M
\\
L_m= 0
&\forall x \in \pd \M
\end{cases}
    \end{aligned}
\ee
where 
$$\tilde{V}_m(x)= V_m(x)\, W_p(x)^{m-3}.$$ 
Recall from Section~\ref{section_lin} that $H_{m}$ is explicitly known as it only depends on $V_1,\ldots,V_{m-1}$ and these functions are known from the induction hypothesis. Applying Green's identity we observe that
\bel{}
\begin{aligned}
&-\int_{\pd \M} f_{\rho}^{-}\, (\pd_{\nu}L_m)\,d\sigma_{\g}+\int_{\pd \M} (\pd_{\nu}\mathcal U_{\rho}^{-})\, L_m\,d\sigma_\g \\
&=\int_{\M} \mathcal U_{\rho}^{-} (\mathcal P_{V_1} L_m)\,dV_{\g} - \int_{\M} (\mathcal P_{V_1} \mathcal U_{\rho}^{-}) L_m\,dV_{\g}\\
&=\int_{\M}\tilde{V}_m\,(\mathcal U_{\rho}^{-}\mathcal U_\rho^{+})^2\,dV_{\g}+\int_{\M}H_m\,\mathcal U_\rho^-\,dV_{\g}. 
\end{aligned}
\ee
We deduce that the map $\Lambda_V$ uniquely determines the expression
$$\int_{\M}\tilde{V}_m\, (\mathcal U_\rho^+)^2(\mathcal U_\rho^-)^{2}\,dV_{\g}.$$
Thus using the same asymptotic analysis as in the proof of Proposition~\ref{V_3}, together with Proposition~\ref{jacobi ray inversion}, we conclude that $\Lambda_V$ uniquely determines $\tilde{V}_m(p)$ and consequently by Lemma~\ref{nonvanishing} it determines 
the function $V_m$ at the all points $(x^0,p)\in \M$ with $p \in \mathscr T$. The proof is completed by noting the density of $\mathscr T$ in $M$.
\end{proof}

\section{Proof of Theorem 2}
\label{section_t2}

\begin{proof}[Proof of Theorem 2]
It suffices to prove that $V_2$ can be uniquely reconstructed from the knowledge of $\Lambda_V$. Indeed, the recovery of the higher order derivatives $V_m$ with $m=3,4,\ldots$, follows by using induction as in the proof of Theorem~\ref{t1}.

We start by choosing an arbitrary point $p \in M$ and assume that $\gamma$ is a maximal non-self intersecting geodesic passing through $p$. Let $\rho=\lambda+i\sigma$ with $\sigma$ fixed and construct a family of complex geometric optic solutions $\mathcal U_\rho^{\pm}$ and $\mathcal U_{2\rho}^{\pm}$ as in Section~\ref{cgo_section} with a decay rate given by an integer $s \geq \frac{n}{6}$ where we recall that $n \in \{3,4\}$. Here, we have assigned the initial values for the ODEs that govern the phase function $\Theta(y')$ at the point $p$. Let us now define 
$$f^{\pm}_\rho=\mathcal U_\rho^{\pm}|_{\pd \M}.$$
Recall that $\mathcal U^{\pm}_\rho \in \CI^{3}(\M)$ and that $\mathcal U^{\pm}_\rho|_{\pd \M}$ are explicitly known. By using the definition of $f_{\rho}^{\pm}$ 
together with the arguments in Section~\ref{section_lin}, we deduce that we know $\pd_\nu L_{f_\rho^+,f_\rho^+}|_{\pd \M}$ for $L_{f_\rho^+,f_\rho^+}$ solving the equation
\bel{eq_lin_2}
\begin{aligned}
\begin{cases}
\mathcal P_{V_1}\,L_{f_\rho^+,f_\rho^+}\,=V_2\, (\mathcal U_\rho^+)^2, 
&\forall x \in \M
\\
L_{f_\rho^+,f_\rho^+}= 0
&\forall x \in \pd \M
\end{cases}
    \end{aligned}
\ee
Let $d\sigma_{\g}$ denote the volume form on $\pd \M$. Applying Green's identity we observe that
\bel{}
\begin{aligned}
&\int_{\pd \M} f_{2\rho}^{-}\, (\pd_{\nu}L_{f_\rho^+,f_\rho^+})\,d\sigma_{\g}-\int_{\pd \M} (\pd_{\nu}\mathcal U_{2\rho}^{-})\, L_{f_\rho^+,f_\rho^+}\,d\sigma_\g \\
=&\int_{\M} \mathcal U_{2\rho}^{-} (\mathcal P_{V_1} L_{f_\rho^+,f_\rho^+})\,dV_{\g} - \int_{\M} (\mathcal P_{V_1} \mathcal U_{2\rho}^{-}) L_{f_\rho^+,f_\rho^+}\,dV_{\g}\\
=&\int_{\M}V_2\,\mathcal U_{2\rho}^{-}(\mathcal U_\rho^{+})^2\,dV_{\g}. 
\end{aligned}
\ee
Thus we can conclude that for all $\lambda>\lambda_0$ with $\lambda^2 \notin \{\mu_l\}_{l \in \N}$, the knowledge of the DN map $\Lambda_V$ uniquely determines the expression 
\bel{}
S_1=\lim_{\lambda\to \infty} \lambda^{\frac{n-2}{2}} \int_{\M}V_2\,\mathcal U_{2\rho}^{-} (\mathcal U_\rho^{+})^2\,dV_{\g}. 
\ee
Applying Proposition~\ref{cgo_sol} and Lemma~\ref{quasi_est}, we have the bounds
\[
\begin{aligned}
 \int_{\M} |V_2||\mathcal R_\rho^+|^2|\mathcal R_{2\rho}^-|\,dV_{\g} &\lesssim |\lambda|^{-3s} \lesssim|\lambda|^{-\frac{n-2}{2}}\lambda^{-1},\\
\int_{\M}|V_2||\mathcal V_\rho^+|^2|\mathcal R_{2\rho}^-|\, dV_{\g} &\lesssim  \lambda^{-\frac{n-2}{2}}\lambda^{-s},\\
\int_{\M}|V_2||\mathcal V_{2\rho}^-||\mathcal V_\rho^+||\mathcal R_{\rho}^+|\, dV_{\g} &\lesssim  \lambda^{-\frac{n-2}{2}}\lambda^{-s},\\
\int_{\M}|V_2||\mathcal V_{2\rho}^-||\mathcal R_{\rho}^+|^2\, dV_{\g} &\lesssim  \lambda^{-\frac{n-2}{2}}\lambda^{-2s}.
\end{aligned}
\]
Next and by using the form of the phase function and the amplitude functions given by equations \eqref{phase_0} and \eqref{amplitude} together with the previous four bounds, we write
$$ S_1 = \lim_{\lambda\to \infty}\lambda^{\frac{n-2}{2}}\int_I\int_{\mathcal N_\gamma} V_2 e^{4i\sigma x^0} e^{-4\lambda\Im \Theta}e^{-4\sigma\Re \Theta} |v_{00}|^2\,v_{00}\,dx^0\,dV_{g}$$
where it is to be understood that $V_2$ is extended to $I\times M$ by setting it to zero outside $\M$. Finally, applying stationary phase together with equations \eqref{stationary_ph_0} and \eqref{amplitude_principal_1}, we deduce that the Dirichlet to Neumann map determines the expression
\bel{S_1_f}\int_{\tau_-}^{\tau_+} e^{\xi t}\mathscr F V_2(\xi,\gamma(t)) \,(\det Y(t))^{-\frac{1}{2}}\,dt,\quad \forall\,|\xi|<\frac{\sigma_0}{4},\ee
with $\mathscr F V_2$ denoting the Fourier transform of $V_2$ with respect to $x^0$ variable. Here, we have extended the function $V_2$ to the entire $\R \times M$  by setting it to zero outside of $\M$. We can summarize the analysis thus far as follows. For each point $p \in M$ and each admissible geodesic $\gamma$ passing through $p$, we have shown that the Dirichlet-to-Neumann map $\Lambda_V$ uniquely determines the expression \eqref{S_1_f}. Repeating the same analysis, this time with the dual choices $L_{f^-_\rho,f^-_\rho}$ and $\mathcal U^{+}_{2\rho}$ we can similarly determine the expression
\bel{S_2_f}\int_{\tau_-}^{\tau_+} e^{\xi t}\mathscr F V_2(\xi,\gamma(t)) \,\overline{(\det Y(t))^{-\frac{1}{2}}}\,dt,\quad \forall\,|\xi|<\frac{\sigma_0}{4}.\ee
Finally, combining the expressions \eqref{S_1_f}--\eqref{S_2_f} we deduce that the map $\Lambda_V$ uniquely determines the expressions
\bel{S} S=\int_{\tau_-}^{\tau_+} e^{\xi t}f(\xi,\gamma(t)) \,(\det Y(t))^{-\frac{1}{2}}\,dt,\quad |\xi|<\frac{\sigma_0}{4},\ee
where $ f \in \{ \Re \mathscr F V_2, \Im \mathscr F V_2\}$.

Applying Proposition~\ref{jacobi ray inversion 1} to the real-valued function $f$ together with the fact that $\sigma_0$ is arbitrary, it immediately follows that $f(\xi,p)$ can be uniquely determined from $\Lambda_V$ for all $\xi \in \R$ and all $p \in M$. We conclude that $f(\xi,x')$ can be uniquely determined from $\Lambda_V$ for all $x' \in M$. Using the inverse Fourier transform, it follows that $\Lambda_V$ uniquely determines the function $V_2$ everywhere in $\M$. 
\end{proof}

\medskip\paragraph{\bf Acknowledgments}
A.F was supported by EPSRC grant EP/P01593X/1. L.O acknowledges support from EPSRC grants EP/R002207/1 and EP/P01593X/1.


\end{document}